\documentclass[11pt,leqno]{amsart}
\usepackage{times,amsmath,amsthm,amssymb,latexsym,amscd}
\usepackage{tikz-cd}
\usepackage[all]{xy}
\usepackage{graphicx}
\usepackage{enumerate}
\usepackage{setspace}
\newtheorem{thm}{Theorem}[section]
\newtheorem{cor}[thm]{Corollary}
\newtheorem{lem}[thm]{Lemma}
\newtheorem{lemma}[thm]{Lemma}
\newtheorem{prop}[thm]{Proposition}

\newtheorem{definition}[thm]{Definition}

\newtheorem{ex}[thm]{Example}

\newcommand{\propref}[1]{Prop\-o\-si\-tion~\ref{#1}}
\newcommand{\corref}[1]{Cor\-ol\-lary~\ref{#1}}

\let\abs=\envert

\newcommand{\lqp}{L(q:p_1, \ldots, p_n)}

\newtheoremstyle{fancy}{}{}{\itshape}{}{\textsc\bgroup}{.\egroup}{ }{}
\newtheoremstyle{fanci}{}{}{\rm}{}{\textsc\bgroup}{.\egroup}{ }{}
\theoremstyle{fancy}
\newcounter{intro}

\numberwithin{equation}{section} 

\theoremstyle{fanci}

\newcommand{\cref}[1]{Corollary~\ref{#1}}

\let\abs=\envert

\begin{document}

\title[Isospectrality for Orbifold Lens Spaces]{Isospectrality for Orbifold Lens Spaces}

\author[N. S. Bari and E. Hunsicker]{Naveed S. Bari and Eugenie Hunsicker}
\noindent \address{Naveed S. Bari \\ 47 Lloyd Wright Avenue, Manchester M11 3NJ, United Kingdom}
\noindent \email{bari.naveed@yahoo.com}
\noindent \address{Eugenie Hunsicker \\ Department of Mathematics, Loughborough University, Loughborough, LE11 3TU, United Kingdom}
\noindent \email{E.Hunsicker@lboro.ac.uk}
\thanks{{\it Keywords:} Spectral geometry \ Global Riemannian
  geometry \ Orbifolds \ Lens Spaces} 
\thanks
{2000 {\it Mathematics Subject Classification:}
Primary 58J53; Secondary 53C20.}


\begin{abstract}
We answer Mark Kac's famous question \cite{K}, ``can one
hear the shape of a drum?'' in the positive
for orbifolds that are 3-dimensional and 4-dimensional lens spaces; 
we thus complete the answer to this question for orbifold lens spaces 
in all dimensions. We also show that the coefficients of the asymptotic 
expansion of the trace of the heat kernel are not sufficient to determine the above
results.

\end{abstract}

\maketitle

\tableofcontents

\section{Introduction}\label{introduction}

Given a closed Riemannian manifold $(M,g)$, the eigenvalue spectrum of the 
associated Laplace Beltrami operator will be referred to as the spectrum of 
$(M,g)$.  The inverse spectral problem asks the extent to which the spectrum 
encodes the geometry of $(M,g)$.  While various geometric invariants such as 
dimension, volume and total scalar curvature are spectrally determined, 
numerous examples of isospectral Riemannian manifolds, i.e., manifolds with 
the same spectrum, show that the spectrum does not fully encode the geometry.  
Not surprisingly, the earliest examples of isospectral manifolds were manifolds 
of constant curvature including flat tori (\cite{M}), hyperbolic manifolds (\cite{V}), 
and spherical space forms (\cite{I1}, \cite{I2} and \cite{Gi}).  In particular, lens 
spaces are quotients of round spheres by cyclic groups of orthogonal transformations 
that act freely on the sphere.  Lens spaces have provided a rich source of isospectral 
manifolds with interesting properties.  In addition to the work of Ikeda and Yamamoto cited above, 
see the results of Gornet and McGowan \cite{GoM}.

In this paper we generalize this theme to the category of Riemannian orbifolds. 
A smooth \emph{orbifold} is a topological space that is locally modelled on an
orbit space of $\mathbf{R}^{n}$ under the action of a finite group of diffeomorphisms. 
\emph{Riemannian} orbifolds are spaces that are locally modelled on quotients of Riemannian 
manifolds by finite groups of isometries. Orbifolds have wide applicability, for example, in 
the study of 3-manifolds and in string theory \cite{DHVW}, \cite{ALR}.

The tools of spectral geometry can be transferred to the setting of Riemannian
orbifolds by using their well-behaved local structure (see \cite{Chi}, \cite{S1} \cite{S2}).
As in the manifold setting, the spectrum of the Laplace operator of a compact
Riemannian orbifold is a sequence $0 \leq \lambda_{1} \leq \lambda_{2} \leq \lambda_{3} \leq \ldots 
\uparrow \infty$ where each eigenvalue is repeated according to its finite multiplicity. 
We say that two orbifolds are isospectral if their Laplace spectra agree.

The literature on inverse spectral problems on orbifolds is less developed than that for manifolds.  
Examples of isospectral orbifolds include pairs with boundary (\cite{BCDS} and \cite{BW}); isospectral 
flat 2-orbifolds (\cite{DR}); arbitrarily large finite families of isospectral orbifolds (\cite{SSW}); 
isospectral orbifolds with different maximal isotropy orders (\cite{RSW}); and isospectral deformation 
of metrics on an orbifold quotient of a nilmanifold (\cite{PS}).

In the study of inverse isospectral problem, spherical space forms provide a rich 
and important set of orbifolds with interesting results. For the 2-dimensional case, it
is known \cite{DGGW} that the spectrum determines the spherical orbifolds of constant 
curvature \textit{R} $> 0$. In \cite{L}, Lauret found examples in dimensions 5 through 8 
of orbifold lens spaces (spherical orbifold spaces with cyclic fundamental groups) that 
are isospectral but not isometric. For dimension 9 and higher, the author proved the existence
of isospectral orbifold lens spaces that are non-isometric \cite{Ba}. The problem was unsolved
for 3 and 4-dimensional orbifold lens spaces. For 3-dimensional manifold lens spaces Ikeda and 
Yamamoto (see \cite{I1}, \cite{IY} and \cite{Y})proved that the spectrum determines the lens space. 
In \cite{I2}, Ikeda further proved that for general 3-dimensional manifold spherical 
space forms, the spectrum determines the space form. In the manifold case, it is also known that even 
dimensional spherical space forms are only the canonical sphere and the real projective space. For 
orbifold spherical space forms this is not the case. In this article we will prove the following results:

\smallskip
\noindent \textbf{Theorem \ref{thm:6}}
Two three-dimensional isospectral orbifold lens spaces are isometric.

\smallskip
\noindent \textbf{Theorem \ref{thm:7}}
Two four-dimensional isospectral orbifold lens spaces are isometric.

\smallskip
\noindent \textbf{Theorem \ref{thm:8}}
Let $\mathbb{S}^{2n-1}/G$ and $\mathbb{S}^{2n-1}/G'$ be two (orbifold) spherical space forms. Suppose $G$ is cyclic and $G'$ is not cyclic. Then $\mathbb{S}^{2n-1}/G$ and $\mathbb{S}^{2n-1}/G'$ cannot be isospectral.

\smallskip
\noindent The above results will complete the classification of the inverse spectral problem on 
orbifold lens spaces in all dimensions.

In addition to the above theorems, we also prove that the coefficients of the trace of the heat kernel are not sufficient to
prove the above results, i.e., we can have two non-isospectral orbifold lens spaces with identical coefficients of the trace of heat kernel.

\section{Orbifold Lens Spaces}\label{orblenschapt}

In this section we will generalize the idea of manifold lens spaces to orbifold lens spaces. Note that 
lens spaces are special cases of \textit{spherical space forms}, which are \textit{connected complete 
Riemannian manifolds of positive constant curvature 1}. An n-dimensional spherical space form can be 
written as $\mathbb{S}^n/G$ where $G$ is a finite subgroup of the orthogonal group $O(n + 1)$. In fact, 
the definition of spherical space forms can be generalized to allow $G$ to have fixed points making 
$\mathbb{S}^n/G$ an orbifold. Manifold lens spaces are spherical space forms where the $n$-dimensional 
sphere $\mathbb{S}^n$ of constant curvature $1$ is acted upon by a cyclic group of fixed point free 
isometries on $\mathbb{S}^n$. We will generalize this notion to orbifolds by allowing the cyclic group 
of isometries to have fixed points. For details of spectral geometry on orbifolds, see Stanhope \cite{S1} 
and E. Dryden, C. Gordon, S. Greenwald and D. Webb in \cite{DGGW}). 

\subsection{Orbifold Lens Spaces and their Generating Functions}\label{subsection31} 

We now reproduce the background work developed by Ikeda in \cite{I1} and \cite{I2} for manifold spherical space forms. We will note that with slight modifications the results are valid for orbifold spherical space forms. This is the background work we will need to develop our results for orbifold lens spaces.

We will first consider general $2n - 1$ dimensional lens spaces. Let $q$ be a positive integer. Set \[ q_{0} = \begin{cases}
													\frac{q-1}{2} & \qquad \text{if } q \text{ is odd,} \\
													\text{ } \frac{q}{2} & \qquad \text{if } q \text{ is even}.
											    \end{cases} \]
Throughout this section we assume that $q_0 \geq 4$. 

For $n \leq q_0$, let $p_1,\ldots , p_n$ be $n$ integers. Note, if $g.c.d.(p_1,\dots, p_n, q) \ne 1$, we can divide all the $p_i's$ and $q$ by this $gcd$ to get a case where the $gcd = 1$. So, without loss of generality, we can assume $g.c.d.(p_1,\dots, p_n, q) =1$. We denote by $g$ the orthogonal matrix given by 
\[ g = 
\begin{pmatrix}
 R(p_{1} / q) & & \text{ {\huge 0}} \\
  & \ddots & \\
  \text{ {\huge 0}} & & R(p_{n} / q)
\end{pmatrix}
,\]
where 
\begin{equation}\label{eqmatrix}
R(\theta ) = \begin{pmatrix} \cos 2\pi \theta & \sin 2\pi \theta \\ -
\sin 2\pi \theta & \cos 2\pi \theta \end{pmatrix}.
\end{equation}
Then $g$ generates a cyclic subgroup $G = \big \{ g^{l} \big \}_{l=1}^{q}$ of
order $q$ of the special orthogonal group $SO(2n)$ since $\det{g} = 1$. Note that $g$ has eigenvalues $\gamma^{p_1}$, $\gamma^{-p_1}$,$\gamma^{p_2}$, $\gamma^{-p_2}$,..., $\gamma^{p_n}$, $\gamma^{-p_n}$, where $\gamma$ is a primitive $q$-th root of unity. We define the lens space $L(q : p_{1}, \ldots, p_n)$ as follows: 
\[ L(q : p_{1}, \ldots, p_n) = S^{2n-1} / G .\] 

Note that if $gcd(p_i, q) = 1$ $\forall i$, $L(q : p_{1}, \ldots, p_n)$ is a smooth manifold; Ikeda and Yamamoto have answered Kac's question in the affirmative for 3-dimensional manifold lens spaces (\cite{IY}, \cite{Y}). To get an orbifold in this setting with non-trivial singularities, we must have $gcd(p_i, q) > 1$ for some $i$. In such a case $L(q : p_{1}, \ldots, p_n)$ is a good smooth orbifold with $S^{2n-1}$ as its
covering manifold. Let $\pi$ be the covering projection of $\mathbb{S}^{2n-1}$ onto $\mathbb{S}^{2n-1} / G$ 
\[ \pi : \mathbb{S}^{2n-1} \rightarrow \mathbb{S}^{2n-1} / G .\] 
Since the round metric of constant curvature one on $\mathbb{S}^{2n-1}$ is
$G$-invariant, it induces a Riemannian metric on $\mathbb{S}^{2n-1}/G$. Henceforth, the term 
''lens space'' will refer to this generalized definition.
\smallskip
Ikeda proved the following result for manifold spherical space forms. We note that the proof doesn't require the 
groups to be fixed-point free, and reproduce the result for orbifold spherical space forms:

\begin{lem}\label{lemma31}
Let $\mathbb{S}^n/G$ and $\mathbb{S}^n/G'$ be spherical space forms for any integer $n \geq 2$. Then $\mathbb{S}^n/G$ is isometric to $\mathbb{S}^n/G'$ if and only if $G$ is conjugate to $G'$ in $O(n + 1)$.
\end{lem}

\smallskip

Note that if we have a lens space $\mathbb{S}^{2n-1}/G = L(q : p_{1}, \ldots, p_n)$, with $G = <g>$, permuting the $p_i$'s doesn't change the underlying group G; similarly, if we multiply all the $p_i$'s by some number $\pm l$ where $gcd(l, q) = 1$, that simply means we have mapped the generator $g$ to the generator $g^l$, and so we still have the same group $G$. Also note that if two lens spaces $\mathbb{S}^{2n-1}/G = L(q : p_{1}, \ldots, p_n)$ and $\mathbb{S}^{2n-1}/G' = L(q : s_{1}, \ldots, s_n)$ are isometric, then by the above lemma $G$ and $G'$ must be conjugate. So, the lift of the isometry on $\mathbb{S}^{2n-1}$ maps a generator, $g$ of $G$ to a generator ${g'}^l$ of $G'$. This means that the eigenvalues of $g$ and ${g'}^l$ are the same, which means that each $p_i$ is equivalent to some $ls_j$ or $-ls_j$ $(\text{mod } q)$. These facts give us the following corollary for Lemma \ref{lemma31}

\begin{cor} \label{lensspaceisometric}
Let $L = L(q : p_{1}, \ldots, p_n)$ and $L' = L(q : s_{1}, \ldots, s_n)$ be
lens spaces. Then $L$ is isometric to $L'$ if and only if there is a number 
$l$ coprime with $q$ and there are numbers $e_{i} \in \{-1,1 \}$ such that $(p_1, \ldots, p_n)$
is a permutation of $(e_{1}ls_1, \ldots, e_{n}ls_n) \pmod{q}$. 
\end{cor}

Assume we have a spherical space form $\mathbb{S}^m/G$ for any integer $m \geq 2$. For any $f \in C^{\infty} (\mathbb{S}^m/G)$, we define the Lapacian on the spherical space form as
$\widetilde{\Delta}(\pi^{*} f) = \pi^{*} (\Delta f)$. 
We now construct the spectral generating function associated with the Laplacian on $\mathbb{S}^{2n-1}/G$ analogous to the construction in the manifold case (see \cite{I1}, \cite{I2} and \cite{IY}). Let $\tilde{\Delta}$, $\Delta$ and $\Delta_0$ denote the Laplacians of $\mathbb{S}^{2n-1}$, $\mathbb{S}^{2n-1}/G$ and $\mathbb{R}^{2n}$, respectively. 
 
 \begin{definition} For any non-negative real number $\lambda$, we define the
\emph{eigenspaces} $\widetilde{E}_{\lambda}$ and $E_{\lambda}$ as follows: 
 \begin{align*}
 	\widetilde{E}_{\lambda} &= \big \{ f \in C^{\infty} (\mathbb{S}^{2n-1}) \big \arrowvert
\widetilde{\Delta} f = \lambda f \big \} ,\\ 
	E_{\lambda} &= \big \{ f \in C^{\infty} (\mathbb{S}^{2n-1}/G) \big \arrowvert \Delta f =
\lambda f \big \}.
\end{align*} 
\end{definition} 

\noindent The following lemma follows from the definitions of $\Delta$ and smooth function. 

\begin{lem}
Let $G$ be a finite subgroup of $O(n+1)$. 
\begin{enumerate}
\item[(i)] For any $f \in C^{\infty} (\mathbb{S}^{2n-1}/G)$, we have
$\widetilde{\Delta}(\pi^{*} f) = \pi^{*} (\Delta f)$. 
\item[(ii)] For  any $G$-invariant function $F$ on $\mathbb{S}^{2n-1}$, there exists a
unique function $f \in C^{\infty} (\mathbb{S}^n/G)$ such that $F = \pi^{*} f$. 
\end{enumerate}
\end{lem} 

\begin{cor} \label{ch3dimcorollary}
Let $\big (\widetilde{E}_{\lambda} \big )_{G}$ be the space of all $G$-invariant
functions of $\widetilde{E}_{\lambda}$. Then $\dim (E_{\lambda}) = 
\dim (\widetilde{E})_{G}$. 
\end{cor}

Let $\Delta_0$ be the Laplacian on $\mathbf{R}^{2n}$ with respect to the flat
K\"ahler metric. Set $r^2 = \sum_{i=1}^{2n} x_{i}^2$, where $(x_1, x_2, \ldots, 
x_{2n})$ is the standard coordinate system on $\mathbf{R}^{2n}$. For $k \geq 0$, 
let $P^k$ denote the space of complex valued homogeneous polynomials of degree $k$ 
on  $\mathbb{R}^{2n}$. Let $H^k$ be the subspace of $P^k$ consisting of harmonic 
polynomials on $\mathbb{R}^{2n}$, \[ H^k = \big \{ f \in P^k \big \arrowvert \Delta_{0}f = 0 \big \} .\]
Each orthogonal transformation of $\mathbb{R}^{2n}$ canonically induces a linear isomorphism of $P^k$.  

\begin{prop} \label{ch3inducedlinmap}
The space $H^k$ is $O(2n)$-invariant, and $P^k$ has the direct sum
decomposition: $P^k = H^k \oplus r^{2} P^{k-2}$. 
\end{prop}

\noindent The injection map $i: \mathbb{S}^{2n-1} \rightarrow \mathbb{R}^{2n}$ induces a linear map
$i^{*}: C^{\infty}(\mathbb{R}^{2n}) \rightarrow C^{\infty}(\mathbb{S}^{2n-1})$. 
We denote $i^{*}(H^k)$ by $\mathcal{H}^k$.

\begin{prop} \label{ch3eigenspaceprop}
$\mathcal{H}^k$ is an eigenspace of $\widetilde{\Delta}$ on $\mathbb{S}^{2n-1}$ with
eigenvalue $k(k+2n-2)$ and $\sum_{k=0}^{\infty} \mathcal{H}^{k}$ 
is dense in $C^{\infty}(\mathbb{S}^{2n-1})$ in the uniform convergence topology.
Moreover, $\mathcal{H}^k$ is isomorphic to $H^k$. That is, 
$i^{*}\! : H^k \stackrel{\simeq}{\longrightarrow} \mathcal{H}^k$. 
\end{prop}
\noindent For proofs of these propositions, see \cite{BGM}. 

Now Corollary \ref{ch3dimcorollary} and \propref{ch3eigenspaceprop} imply that if we denote by 
$\mathcal{H}_{G}^k$ be the space of all $G$-invariant functions in
$\mathcal{H}^k$, then
\[ \dim E_{k(k+2n-2)} = \dim \mathcal{H}_{G}^k .\]
Moreover, for any integer $k$ such that $\dim \mathcal{H}_{G}^k \neq 0$,
$\bar{\lambda}_k = k(k+2n-2)$ is an eigenvalue of $\Delta$ on $\mathbb{S}^{2n-1}/G$ with multiplicity equal to 
$\dim \mathcal{H}_{G}^k$, and no other eigenvalues appear in the spectrum of
$\Delta$. 

\begin{definition}
Let $O$ be a closed compact Riemannian orbifold with the Laplace spectrum,  $0 \leq
\bar{\lambda}_{1} < \bar{\lambda}_{2} < \bar{\lambda}_{3} \ldots \uparrow \infty$. For each $\bar{\lambda}_k$, let the eigenspace be  
 \begin{align*}
 		E_{\bar{\lambda}_k} &= \big \{ f \in C^{\infty} (O) \big \arrowvert \Delta f =
\bar{\lambda}_k f \big \}.
\end{align*} 
We define the spectrum generating function associated to the spectrum of the Laplacian on $O$ as 
\[F_{O}(z) = \sum_{k =0}^{\infty} \big( \dim {E_{\bar{\lambda}_k}} \big) z^k .\] 
\end{definition}

\noindent In terms of spherical space forms, the definition becomes

\begin{definition}
The \emph{generating function} $F_{G}(z)$ associated to the
spectrum of the Laplacian on $\mathbb{S}^n/G$ is the generating function 
associated to the infinite sequence $\big \{ \dim \mathcal{H}_{G}^{k} \big
\}_{k=0}^{\infty}$ , i.e., 
\[ F_{G}(z) = \sum_{k =0}^{\infty} \big( \dim
\mathcal{H}_{G}^{k} \big) z^k .\] 
\end{definition}

By \corref{ch3dimcorollary}, Proposition \ref{ch3eigenspaceprop} and subsequent discussion, we know that the generating function determines the spectrum of $\mathbb{S}^n/G$. This fact gives us the following proposition:

\begin{prop} \label{prop329}
Let $\mathbb{S}^n/G$ and $\mathbb{S}^n/G'$ be two spherical space forms. Let
$F_{G}(z)$ and $F_{G'}(z)$ be their respective spectrum generating functions. Then $\mathbb{S}^n/G$ is
isospectral to $\mathbb{S}^n/G'$ if and only if $F_{G}(z) = F_{G'}(z)$. 
\end{prop}

\noindent Our first goal is to find an alternative expression for $F_{G}(z)$ that will allow us to compare $F_{G}(z)$ and $F_{G'}(z)$.  

If $G$ is a finite subgroup of $O(2n)$ with orientation preserving action on $S^{2n-1}$ then $G$ is a subgroup of $SO(2n)$. In the following we will consider orientation-preserving group actions.

The following theorem, proved for manifold spherical space forms in \cite{I1} and \cite{I2}, holds true for the orbifold spherical space forms as well.

\begin{thm} \label{theorem3209}
Let $G$ be a finite subgroup of $SO(2n)$, and let $\mathbb{S}^{2n-1}/G$  be a spherical space form with spectrum generating function $F_G(z)$. Then, on the domain 
$\big \{ z \in \mathbb{C} \, \big \arrowvert  \, |z| < 1 \big \}$, $F_G(z)$ converges to the function
\[ F_{G}(z) = \frac{1}{|G|} \, \sum_{g \in G}\,
\frac{1-z^2}{\text{det}(I_{2n} - gz)}.\]
where $|G|$ denotes the order of $G$ and $I_{2n}$ is the $2n \times 2n$ identity matrix.
\end{thm}

\noindent We denote the generating function for a lens space $L = \lqp$ by $F_q(z: p_1, ..., p_n)$.

\begin{cor} \label{theorem3210}
Let $\lqp$ be a lens space and $F_{q}(z: p_1, \ldots, p_n)$ the generating
function associated to the spectrum of $\lqp$. Then, on the domain 
$\big \{ z \in \mathbb{C} \, \big \arrowvert  \, |z| < 1 \big \}$, 
\[ F_{q}(z: p_1, \ldots, p_n) = \frac{1}{q} \, \sum_{l=1}^{q} \,
\frac{1-z^2}{\prod_{i=1}^{n} (z - \gamma^{p_{i} l} )(z - \gamma^{-p_{i} l})} ,\]
where $\gamma$ is a primitive $q$-th root of unity.
\end{cor}
\begin{proof}
In the notation of the Theorem \ref{theorem3209}, we get
\begin{align}\label{eq4.5}
\dim \mathcal{H}_{G}^k = \frac{1}{|G|} \sum_{g \in G}  \chi_{k}(g) = \frac{1}{q}
\sum_{l=1}^{q} \chi_{k} (g^l).
\end{align}

So 
\begin{align*}
F_{q}(z: p_1, \ldots, p_n) &= \frac{(1-z^2)}{|G|} \sum_{g \in G} \, 
\frac{1}{\prod_{i=1}^{n} (1 - \gamma^{p_i}z) (1 - \gamma^{-p_i}z)}\\
&= \frac{(1-z^2)}{q} \sum_{l=1}^{q} \, 
\frac{1}{\prod_{i=1}^{n} (z - \gamma^{p_i l}) (z - \gamma^{-p_i l})},
\end{align*}
since multiplying through by $1 = (-\gamma^{-p_i l})(-\gamma^{p_i l})$ 
gives \\ $(1 - \gamma^{p_i l} z) (1 - \gamma^{-p_i l} z) = (z - \gamma^{-p_i l}) (z
- \gamma^{p_i l})$.
\end{proof}
 
\textbf{Remark}: By the Theorem \ref{theorem3209} and unique analytic continuation, we can consider the generating function to be a meromorphic function on the whole complex plane $\mathbb{C}$ with poles on the unit circle $\mathbb{S}^1 = \{z\in \mathbb{C} \mid |z| = 1\}$. 

From this remark we have,

\begin{cor}
Let $\mathbb{S}^{2n-1}/G$ and $\mathbb{S}^{2n-1}/G'$ be two spherical space forms. If there is a one to one mapping $\phi$ of $G$ onto $G'$ such that the set $E(g)$ = the set $E(\phi(g)), \forall g \in G$, then $\mathbb{S}^{2n-1}/G$ is isospectral to $\mathbb{S}^{2n-1}/G'.$
\end{cor}

\begin{proof}
The proof follows from the fact that
\[\prod_{\gamma \in E(g)} (1 - \gamma z) 
= \prod_{\gamma \in E(g)} (z - \gamma) 
= \text{det} (I_{2n} - gz).\]
\end{proof}

\begin{cor}\label{corG}
Let $\mathbb{S}^{2n-1}/G$ and $\mathbb{S}^{2n-1}/G'$ be two isospectral spherical space forms. Then $|G| = |G'|$. 
\end{cor}

\section{3-Dimensional Orbifold Lens Spaces}

For 3-dimensional manifold lens spaces, it is known that if two lens spaces are isospectral then they are also isometric (\cite{IY} and \cite{Y}). We will generalize this result to the orbifold case.

Using the notation adopted in the previous section, we write the two isospectral lens spaces 
as $L_1 = L(q: p_1, p_2)$ and $L_2 = L(q: s_1, s_2)$. Now there are only five possibilities:
\begin{enumerate}
\item[Case 1] Both $L_1$ and $L_2$ are manifolds. In this case $gcd(p_i, q) = 1 = gcd(s_i, q)$ for $i = 1, 2$.
\item[Case 2] One of the two lens spaces, say $L_1$ is a manifold, while the other, $L_2$ is an orbifold with non-trivial isotropy groups. This means that $gcd(p_1, q) = gcd(p_2, q) = 1$, while at least one of $s_1$ or $s_2$ is not coprime to $q$.
\item[Case 3] Both $L_1$ and $L_2$ are orbifolds with non-trivial isotropy groups so that exacly one of $p_1$ or $p_2$ is coprime to $q$ and exactly one of $s_1$ or $s_2$ is coprime to $q$. 
\item[Case 4] Both $L_1$ and $L_2$ are orbifolds with non-trivial isotropy groups, but in one case, say for $L_1$, exactly one of $p_1$ or $p_2$ is coprime to $q$, while for the other lens space, $L_2$ neither $s_1$ nor $s_2$ is coprime to $q$.
\item[Case 5] None of $p_1$, $p_2$, $s_1$ and $s_2$ is coprime to $q$.
\end{enumerate}
\noindent With these five cases in mind, we will prove our main theorem:

\begin{thm}\label{thm:6}
Given two 3-dimensional lens spaces $L_1 = L(q: p_1, p_2)$ and $L_2 = L(q: s_1, s_2)$. If $L_1$ is isospectral to $L_2$, then the two lens spaces are isometric.
\end{thm}

\begin{proof}
We will consider each case separately:
\smallskip

\noindent \textbf{Case 1}

In this case $L_1$ and $L_2$ are both manifolds. Ikeda and Yamamoto proved this case (see \cite{IY} and \cite{Y}).
\smallskip

\noindent \textbf{Case 2}

We know that whenever two isospectral good orbifolds share a common Riemannian cover, their respective singular sets are either both trivial or both non-trivial \cite{GR}. Therefore, for orbifold lens spaces we can't have a situation where two lens spaces are isospectral, but one has a trivial singular set while the other has a non-trivial singular set. So this case is not possible.
\smallskip

\noindent \textbf{Case 3} 

By multiplying the entries of $L_1$ and $L_2$ by appropriate numbers coprime to \textit{q} we can rewrite $L_1 = L(q:1, x)$ and $L_2 = L(q:1, y)$, where $x$ and $y$ are not coprime to $q$. 
Let $F_1(z$) [resp. $F_2(z)]$ be the generating function associated to the spectrum of $L_1$ [resp.$L_2$]. Let $\gamma$ be a primitive $q$-th root of unity.

Now,
\begin{align}\label{cong1}
&\lim\limits_{z\rightarrow {\gamma}}(z - \gamma)F_1(z) \notag \\
& = \lim\limits_{z\rightarrow {\gamma}}\frac{1}{q}\sum\limits_{l=1}^{q}\frac{(z-{\gamma})(1 - z^2)}{(1 - {\gamma}^{l}z)(1 - {\gamma}^{-l}z)(1 - {\gamma}^{xl}z)(1 - {\gamma}^{-xl}z)} \notag\\
&= \lim\limits_{z\rightarrow {\gamma}}\frac{-\gamma}{q}\sum\limits_{l=1}^{q}\frac{(1-{\gamma}^{-1}z)(1 - z^2)}{(1 - {\gamma}^{l}z)(1 - {\gamma}^{-l}z)(1 - {\gamma}^{xl}z)(1 - {\gamma}^{-xl}z)}
\end{align}

\noindent Each term of the sum vanishes unless $(1-{\gamma}^{-1}z)$ cancels one of the four terms in the denominator. This occurs if one of the following congruences has a solution:
\smallskip
\begin{itemize}
\item[(1)] $l + 1 \equiv 0(\text{mod } q),$
\item[(2)] $-l + 1 \equiv 0(\text{mod } q),$
\item[(3)] $xl + 1 \equiv 0(\text{mod } q),$
\item[(4)] $-xl + 1 \equiv 0(\text{mod } q).$ 
\end{itemize}
\smallskip
Congruences (3) and (4) have no solution as $x$ is not coprime to $q$. The solution to (1) is $l = q - 1$, and the solution to (2) is $l = 1$. Substituting in (\ref{cong1}), we get 

\[\lim\limits_{z\rightarrow {\gamma}}(z - \gamma)F_1(z) = \frac{-2\gamma}{q(1 - {\gamma}^{-x+1})(1 - {\gamma}^{x+1})}.\]

By the same argument, we get

\[\lim\limits_{z\rightarrow {\gamma}}(z - \gamma)F_2(z) = \frac{-2\gamma}{q(1 - {\gamma}^{-y+1})(1 - {\gamma}^{y+1})}.\]

Since 
\[\lim\limits_{z\rightarrow {\gamma}}(z - \gamma)F_1(z)=\lim\limits_{z\rightarrow {\gamma}}(z - \gamma)F_2(z),\]  
we get
\begin{align*}
&\frac{-2\gamma}{q(1 - {\gamma}^{-x+1})(1 - {\gamma}^{x+1})} = \frac{-2\gamma}{q(1 - {\gamma}^{-y+1})(1 - {\gamma}^{y+1})},\\
&\implies \frac{1}{[1 - ({\gamma}^{-x+1} + {\gamma}^{x+1}) + {\gamma}^2]} = \frac{1}{[1 - ({\gamma}^{-y+1} + {\gamma}^{y+1}) + {\gamma}^2]},\\
&\implies ({\gamma}^{-x+1} + {\gamma}^{x+1}) = ({\gamma}^{-y+1} + {\gamma}^{y+1}).
\end{align*}

\noindent Since $\gamma \neq 0$, we get
\begin{align*}
&({\gamma}^{-x} + {\gamma}^{x}) = ({\gamma}^{-y} + {\gamma}^{y}),\\
&\implies (\frac{1}{\gamma^x}+\gamma^x) = (\frac{1}{\gamma^y}+\gamma^y),\\
&\implies (\frac{1+\gamma^{2x}}{\gamma^x}) = (\frac{1+\gamma^{2y}}{\gamma^y}),\\
&\implies ({\gamma}^{y} + {\gamma}^{2x + y}) = ({\gamma}^{x} + {\gamma}^{x +2y}),\\
&\implies ({\gamma}^{y} - {\gamma}^{x + 2y}) = ({\gamma}^{x} - {\gamma}^{2x +y}),\\
&\implies \gamma^y(1 - {\gamma}^{x + y}) = \gamma^x(1 - {\gamma}^{x +y}),\\
&\implies (\gamma^y - \gamma^x)(1 - {\gamma}^{x + y}) = 0,\\
&\implies (\gamma^y - \gamma^x) = 0  \text
{ or } (1 - {\gamma}^{x + y}) = 0,\\
&\implies x \equiv y (\text{mod } q) \text
{ or } x \equiv -y (\text{mod } q).
\end{align*}
\noindent Thus, by Corollary \ref{lensspaceisometric} we get that $L_1$ and $L_2$ are isometric.
\smallskip

\noindent \textbf{Case 4}

By the same argument as in Case 3, we get 

\[\lim\limits_{z\rightarrow {\gamma}}(z - \gamma)F_1(z) = \frac{-2\gamma}{q(1 - {\gamma}^{-x+1})(1 - {\gamma}^{x+1})}.\]

\noindent However, 
\[\lim\limits_{z\rightarrow {\gamma}}(z - \gamma)F_2(z) = 0\] because the congruences (1) - (4) in Case 3 become  
\smallskip

\begin{itemize}
\item[(1')] $s_1l + 1 \equiv 0(\text{mod } q),$
\item[(2')] $-s_1l + 1 \equiv 0(\text{mod } q),$
\item[(3')] $s_2l + 1 \equiv 0(\text{mod } q),$
\item[(4')] $-s_2l + 1 \equiv 0(\text{mod } q),$ 
\end{itemize}
\smallskip

\noindent and these congruences have no solutions because $s_1$ and $s_2$ are not coprime to $q$. 

Thus, in this case $L_1$ cannot be isospectral to $L_2$.
\smallskip

\textbf{Case 5}

This is the hardest of all the cases. First, we can simplify the forms of the two lens spaces as follows:

Let $gcd(p_1, q) = x > 1$, $gcd(p_2, q) = y > 1$, $gcd(s_1, q) = u > 1$, and $gcd(s_2, q) = v > 1$. Also without loss of generality we can assume that $y > x$ and $v > u$ because if $x = y$ (resp. $u = v$) then $|G| = q/x$ (resp. $|G| = q/u$), which contradicts our assumption that $|G| = q$. 

We rewrite $L_1 = L(q:ax, by)$ and $L_2 = L(q:cu, dv)$. Since $gcd(ax, q) = gcd(x, q) = x$ and  $gcd(cu, q) = gcd(u, q) = u$, we can multiply the entries of $L_1$ and $L_2$ by appropriate numbers coprime to \textit{q} and rewrite $L_1 = L(q:x, py)$ and $L_2 = L(q:u, sv)$ (see \cite{GP}). We will also assume that $gcd(x, py) = 1 = gcd(u, sv)$ because if say $gcd(x, py) = e > 0$, then we could divide $x$, $py$ and $q$ by \textit{e} and get a lens space with fundamental group of order \textit{q/e} instead of \textit{q}, which is a contradiction.  

In this case we again want to consider a limit of the spectral generating functions for $L_1$ and $L_2$.

\begin{prop}\label{propCase5A}
Suppose $L = L(q:x, py)$ is an orbifold lens space with spectrum generating function $F_q(z)$. 
Then $\lim\limits_{z\rightarrow {\gamma}^{x}}(z - \gamma^{x})F_q(z) \neq 0$, where $\gamma = e^{2\pi i/q}$ is a primitive $q$-th root of unity.
\end{prop}
\begin{proof}
We denote $q_{/x} = \frac{q}{x}$ and $q_{/y} = \frac{q}{y}$. Then
\begin{align}\label{limit1}
\lim\limits_{z\rightarrow {\gamma}^{x}}(z - \gamma^{x})F_q(z) &= \lim\limits_{z\rightarrow {\gamma}^{x}}\frac{1}{q}\sum\limits_{l=1}^{q}\frac{(z-{\gamma}^{x})(1 - z^2)}{(1 - {\gamma}^{xl}z)(1 - {\gamma}^{-xl}z)(1 - {\gamma}^{pyl}z)(1 - {\gamma}^{-pyl}z)} \notag\\
&=\lim\limits_{z\rightarrow {\gamma}^{x}}\frac{-\gamma^x}{q}\sum\limits_{l=1}^{q}\frac{(1-{\gamma}^{-x}z)(1 - z^2)}{(1 - {\gamma}^{xl}z)(1 - {\gamma}^{-xl}z)(1 - {\gamma}^{pyl}z)(1 - {\gamma}^{-pyl}z)}
\end{align}

As before, the terms in the above sum are non-zero iff one of the following congruences has a solution:
\begin{itemize}
\item[(1'')] $xl + x \equiv 0(\text{mod } q),$
\item[(2'')] $-xl + x \equiv 0(\text{mod } q),$
\item[(3'')] $pyl + x \equiv 0(\text{mod } q),$
\item[(4'')] $-pyl + x \equiv 0(\text{mod } q),$ 
\end{itemize}

(3'') implies that  $pyl + x \equiv 0(\text{mod } y)$, so, if (3'') has a solution, it would violate the fact that $gcd(x, y) = 1$. Therefore, (3'') has no solution. Similarly (4'') has no solution. 

The solution to (1'') is $l = tq_{/x} - 1$ and the solution to (2'') is $l = tq_{/x} + 1$ for $t\in \{1, ..., x\}$. Note that for   
 $l = tq_{/x} \pm 1$, 
\[\lim\limits_{z\rightarrow {\gamma}^{x}}\frac{(1-{\gamma}^{-x}z)(1 - z^2)}{(1 - {\gamma}^{xl}z)(1 - {\gamma}^{-xl}z)} = 1\]

 We can, therefore, write (\ref{limit1}) as
\[\lim\limits_{z\rightarrow {\gamma}^{x}}(z - \gamma^{x})F_q(z) = \frac{-2\gamma^x}{q}\sum\limits_{t=1}^{x}\frac{1}{(1 - {\gamma}^{py(tq_{/x} - 1)+ x})(1 - {\gamma}^{-py(tq_{/x} - 1)+ x})} \]

  Writing $\alpha_t = py(tq_{/x} - 1)$, we get 

\begin{align*}
\lim\limits_{z\rightarrow {\gamma}^{x}}(z - \gamma^{x})F_q(z) 
&= \frac{-2\gamma^x}{q}\sum\limits_{t=1}^{x}\frac{1}{(1 - \gamma^{(\alpha_t + x)})(1 - \gamma^{-(\alpha_t - x)}}\\ 
&= \frac{-2\gamma^x}{q}\sum\limits_{t=1}^{x}\frac{1}{\gamma^{(\alpha_t + x)}
(\gamma^{-(\alpha_t - x)} -  \gamma^{-(\alpha_t + x)})}\Big[\frac{1}{1 - \gamma^{-(\alpha_t + x)}} - \frac{1}{1 - \gamma^{-(\alpha_t - x)}}\Big]\\
&= \frac{-2\gamma^x}{q(\gamma^{2x} - 1)}\sum\limits_{t=1}^{x}\Big[\frac{1}{1 - \gamma^{-(\alpha_t + x)}} - \frac{1}{1 - \gamma^{-(\alpha_t - x)}}\Big]\\
&= \frac{-2}{i 2 q \sin{\frac{2\pi x}{q}}}\sum\limits_{t=1}^{x}\Big[\frac{1}{1 - e^{-i2\pi(\alpha_t + x)/q}} - \frac{1}{1 - e^{-i2\pi(\alpha_t - x)/q}}\Big]
\end{align*}

By writing $a_t = \alpha_t + x$ and $b_t = \alpha_t - x$, we can rewrite the above as:
\begin{align*}
\lim\limits_{z\rightarrow {\gamma}^{x}}(z - \gamma^{x})F_q(z)
&= \frac{-2}{i 2 q \sin{\frac{2\pi x}{q}}}\sum\limits_{t=1}^{x}\Big[\frac{1}{1 - e^{-i2\pi a_t/q}} - \frac{1}{1 - e^{-i2\pi b_t/q}}\Big]\\
&= \frac{1}{2q \sin{\frac{2\pi x}{q}}}\sum\limits_{t=1}^{x}\Big[\frac{2i}{1 - e^{-i2\pi a_t/q}} - \frac{2i}{1 - e^{-i2\pi b_t/q}}\Big]
\end{align*}
Now, using the identity $\cot\theta + i = \frac{2i}{1 - e^{-2i\theta}}$, we get
\begin{equation}\label{eq:cot}
\lim\limits_{z\rightarrow {\gamma}^{x}}(z - \gamma^{x})F_q(z) 
= \frac{1}{2q \sin{\frac{2\pi x}{q}}}\sum\limits_{t=1}^{x}\Big[\cot\frac{\pi a_t}{q} - \cot\frac{\pi b_t}{q}\Big].
\end{equation}

The above limit can only be $0$ if 
\begin{align*}
&\sum\limits_{t=1}^{x}\Big[\cot\frac{\pi a_t}{q} - \cot\frac{\pi b_t}{q}\Big] \\
&= \sum\limits_{t=1}^{x}\Big[\cot\frac{\pi}{q}[tpyq_{/x} - (py - x)] - \cot\frac{\pi}{q}[tpyq_{/x} - (py + x)]\Big] = 0.
\end{align*}
Now (mod $q$) both $a_t$ and $b_t$ have $x$ values each between $0$ and $\pi$. 

Consider the following two sets of positive integers (mod $q$):
$$A = \{A_t: A_t = a_t(\text{mod } q)\text{, } t = 1, 2, ..., x\}$$ and $$B = \{B_t: B_t = b_t(\text{mod } q)\text{, } t = 1, 2, ..., x\}.$$

Suppose $min\{A\} = A_j$ and $min\{B\} = B_k$. Now we have the following possibilities:
\begin{itemize}
\item[(i)] $A_j > B_k$. Then it is easy to check that $A_{j+t} > B_{k+t}$ for $t = 0, 1, ..., x-1$, since $a_{j+t} - b_{k+t} = a_j - b_k$. So, we can re-write (\ref{eq:cot}) as 
\begin{equation}\label{eq:cot2}
\lim\limits_{z\rightarrow {\gamma}^{x}}(z - \gamma^{x})F_q(z) 
= \frac{1}{2q \sin{\frac{2\pi x}{q}}}\sum\limits_{t=0}^{x-1}\Big[\cot\frac{\pi}{q}A_{j+t} - \cot\frac{\pi}{q}B_{k+t}\Big].
\end{equation}
We know that if $0 < B < A < \pi$, then $\cot A - \cot B < 0$.  Since, in the above equation $0 < B_{k+t} < A_{j+t} < \pi$ for all $t$, each pair gives us a negative value, and therefore (\ref{eq:cot2}) is negative. 

\item[(ii)] $A_j < B_k$. Then using a similar argument as above, we will have (\ref{eq:cot2}) positive.

\item[(iii)] $A_j = B_k$. This means $a_j - b_k \equiv (j - k)pyq_{/x} + 2x \equiv 0 (\text{mod } q)$. But this means that $y|2x$, which is not possible since we are assuming that $gcd(x, y) = 1$ and $x < y$. 

\end{itemize}
\noindent This proves the proposition.
\end{proof}

We will also need the following lemma to prove the theorem for Case 5:

\begin{lemma} \label{Lemma1}
Suppose $L_1 = L(q:x, py)$ and $L_2 = L(q:u, sv)$ are two isospectral lens orbifolds where $gcd(x, q) = x$, $gcd(py, q) = y$,  $gcd(u, q) = u$ and $gcd(sv, q) = v$. Then either $u = x$ and $v = y$, or $u = y$ and $v = x$.
\end{lemma}

\textbf{Note:} If $u = x$ and $v = y$, then $L_1 = L(q:x, py)$ and $L_2 = L(q:x, sy)$; if $u = y$ and $v = x$, then $L_1 = L(q:x, py)$ and $L_2 = L(q:y, sx) = L(q:s^{-1}y, x) = L(q:x, s^{-1}y) $. In either case, this implies that we can write $L_1 = L(q:x, py)$ and $L_2 = L(q:x, s'y)$ where $s' = s$ or $s' = s^{-1}$.

\noindent We now prove the lemma:

\begin{proof}
We denote $q_{/x} = \frac{q}{x}$ and $q_{/y} = \frac{q}{y}$. Then
\begin{align*}
&\lim\limits_{z\rightarrow {\gamma}^{x}}(z - \gamma^{x})F_1(z) = \lim\limits_{z\rightarrow {\gamma}^{x}}\frac{1}{q}\sum\limits_{l=1}^{q}\frac{(z-{\gamma}^{x})(1 - z^2)}{(1 - {\gamma}^{xl}z)(1 - {\gamma}^{-xl}z)(1 - {\gamma}^{pyl}z)(1 - {\gamma}^{-pyl}z)}
\end{align*}
Recall that the only non-zero terms in this limit will be the ones where $xl+x \equiv 0 (\text{mod }q)$ or $-xl+x \equiv 0 (\text{mod }q)$, which gives $l = tq_{/x} - 1$ or $l = tq_{/x} + 1$ for $t\in \{1, ..., x\}$. Also note that for such a $t$, we have

\begin{align*}
\frac{1}{(1 - \gamma^{py(tq_{/x} - 1)+x})(1 -\gamma^{-py(tq_{/x} - 1)+x})} = \frac{1}{(1 - \gamma^{py[(x-t)q_{/x} + 1]+x})(1 -\gamma^{-py[(x-t)q_{/x} + 1]+x})}.
\end{align*}

\noindent These two facts, along with Proposition \ref{propCase5A} give

\begin{align*}
&0 \neq \frac{-2\gamma^x}{q}\sum\limits_{t=1}^{x}\frac{1}{(1 - \gamma^{py(tq_{/x} - 1) + x})(1 - \gamma^{-py(tq_{/x} - 1) + x})} = \lim\limits_{z\rightarrow {\gamma}^{x}}(z - \gamma^{x})F_1(z).
\end{align*}
\noindent Since 
\[\lim\limits_{z\rightarrow {\gamma}^{x}}(z - \gamma^{x})F_1(z)=\lim\limits_{z\rightarrow {\gamma}^{x}}(z - \gamma^{x})F_2(z),\]
\noindent we get 
\begin{align*}
0 \neq \frac{-2\gamma^x}{q}\sum\limits_{t=1}^{x}\frac{1}{(1 - \gamma^{py(tq_{/x} - 1) + x})(1 - \gamma^{-py(tq_{/x} - 1) + x})} = \lim\limits_{z\rightarrow {\gamma}^{x}}(z - \gamma^{x})F_2(z) \\
= \lim\limits_{z\rightarrow {\gamma}^{x}}\frac{-\gamma^x}{q}\sum\limits_{l=1}^{q}\frac{(1-{\gamma}^{-x}z)(1 - z^2)}{(1 - {\gamma}^{ul}z)(1 - {\gamma}^{-ul}z)(1 - {\gamma}^{svl}z)(1 - {\gamma}^{-svl}z)}.
\end{align*}
\noindent So there must be an $l$ such that
\[ul + x \equiv 0 (\text{mod } q),\]
or 
\[-ul + x \equiv 0 (\text{mod } q),\]
or
\[svl + x \equiv 0 (\text{mod } q),\]
or
\[-svl + x \equiv 0 (\text{mod } q).\]

Recall that $u|q$. Then $ul + x \equiv 0 (\text{mod } q)$ or $-ul + x \equiv 0 (\text{mod } q)$ imply that $u|x$. 
Similarly, since $v|q$, we can show that if $svl + x \equiv 0 (\text{mod } q)$ or  $-svl + x \equiv 0 (\text{mod } q)$ then $v|x$. So either $u|x$ or $v|x$.

Now by multiplying the elements of $L_1$ by an appropriate number we can rewrite 
$L_1 = L(q: y, p'x)$. Then applying the same argument as above where we swap the 
roles of $x$ and $y$, we get either $u|y$ or $v|y$.

Suppose $u|x$. Then since $gcd(x, y) = 1$ we can't have $u|y$. Similarly, if $v|x$, then we can't have $v|y$. Therefore, either $u|x$ and $v|y$, or $v|x$ and $u|y$ since if $u$ or $v$ divide both, then it contradicts $gcs(q, x, py)=1$. 

We can swap the roles of $L_1$ and $L_2$ and repeat the above arguments again to get either $x|u$ and $y|v$, or $y|u$ and $x|v$.

If $u|x$ and $v|y$, and at the same time $x|v$ and $y|u$, then $x|y$, which contradicts the fact that $gcd(q, x, y) = 1$. So, the only possibilities are:
\begin{itemize}
\item[i. ] $u|x$, $v|y$, $x|u$ and $y|v$. This means $x = u$ and $y = v$.
\item[ii. ] $v|x$, $u|y$, $x|v$ and $y|u$. This means $x = v$ and $y = u$.
\end{itemize}

This completes the proof for the lemma.
\end{proof}

\noindent \textbf{Remark:} From now on, we can write the two lens spaces as $L_1 = L(q:x, py)$ and $L_2 = L(q:x, sy)$. Further, If $q$ is odd, we can also assume that both $s$ and $p$ are odd since if one of them, say $p$, is even then we can replace the lens space with $L(q:x, (q-p)y)$ which is isometric to $L_1$ and the coefficient $q-p$ is odd. Also, if $q$ is even, then both $x$ and $py$ (resp.$sy$) can't be even simultaneously since $gcd(x, py)$(resp. $sy$); from now on, without loss of generality, if $q$ is even we will assume that $x$ is even and $py$ (resp. $sy$) is odd since if $py$ (resp. $sy$) is even and $x$ is odd, then we can multiply the entries of the lens spaces by an appropriate number to re-write it as $L_1 = L(q:y, p'x)$ (resp. $L_2 = L(q:y, s'x)$). 
\smallskip

We now returning to the proof of Case 5 of our main theorem. Suppose $L_1 = L(q:x, py)$ and $L_2 = L(q:x, sy)$ are isospectral lens spaces with spectrum generating functions $F_1(z)$ and $F_2(z)$ respectively. Using a similar argument as in Proposition \ref{propCase5A} above and the fact that $F_1(z) = F_2(z)$, we will get 

\begin{align}\label{test1}
&\sum\limits_{t=1}^{x}\Big[\cot\frac{\pi}{q}[tpyq_{/x} - (py - x)] - \cot\frac{\pi}{q}[tpyq_{/x} - (py + x)]\Big] \notag \\
= &\sum\limits_{t=1}^{x}\Big[\cot\frac{\pi}{q}[tsyq_{/x} - (sy - x)] - \cot\frac{\pi}{q}[tsyq_{/x} - (sy + x)]\Big].
\end{align}

Since $py > x$ and $sy > x$ (therefore, $pyq_{/x} > xq_{/x} = q$ and $syq_{/x} > xq_{/x} = q$ respectively), the above equation can be written as 

\begin{align}\label{test2}
&\sum\limits_{t=1}^{x}\Big[\cot\frac{\pi}{q}[tq_{/x} - (py - x)] - \cot\frac{\pi}{q}[tq_{/x} - (py + x)]\Big] \notag \\
= &\sum\limits_{t=1}^{x}\Big[\cot\frac{\pi}{q}[tq_{/x} - (sy - x)] - \cot\frac{\pi}{q}[tq_{/x} - (sy + x)]\Big].
\end{align}

Finally, by writing $\alpha y \equiv (q - p)y(\text{ mod }q)$ and $\beta y \equiv (q - s)y(\text{ mod }q)$, we can rewrite the above equality as  

\begin{align}\label{test3}
&\sum\limits_{t=0}^{x-1}\Big[\cot\frac{\pi}{q}[tq_{/x} + \alpha y + x] - \cot\frac{\pi}{q}[tq_{/x} + \alpha y - x]\Big] \notag \\
= &\sum\limits_{t=1}^{x}\Big[\cot\frac{\pi}{q}[tq_{/x} + \beta y + x] - \cot\frac{\pi}{q}[tq_{/x} +\beta y - x]\Big],
\end{align}
\\

Suppose $tq_{/x} + \alpha y + x > 0$ and $tq_{/x} + \alpha y - x < 0$. But this would mean that $y(tq_{/xy} + \alpha) < x$, which can't be true because we are assuming $y > x$. Therefore, for every $t$, both $tq_{/x} + \alpha y + x$ and $tq_{/x} + \alpha y - x$ are positive(with the only exception happening when $y = q_{/x}$, which we will look at a little later). This observation suggests that the minimum values of $tq_{/x} + \alpha y + x$ and $tq_{/x} + \alpha y - x$ occur for the same value of $t$, and in such a case the difference between the minimum values would be $2x$. The same will be the case for the minimum values of $tq_{/x} + \beta y + x$ and $tq_{/x} + \beta y - x$. 

Now consider the following four sets of positive integers (mod $q$):

$$A = \{A_t: A_t \equiv [tq_{/x} + \alpha y + x](\text{mod } q)\text{, } t = 0, 1, ..., x-1\},$$
$$B = \{B_t: B_t \equiv [tq_{/x} + \alpha y - x](\text{mod } q)\text{, } t = 0, 1, ..., x-1\},$$
$$C = \{C_t: C_t \equiv [tq_{/x} + \beta y - x](\text{mod } q)\text{, } t = 0, 1, ..., x-1\},$$ 
$$D = \{D_t: D_t \equiv [tq_{/x} + \beta y + x](\text{mod } q)\text{, } t = 0, 1, ..., x-1\}.$$
\\

\noindent \textbf{REMARK} Note that the minimum values for $A$ and $B$ (resp. $C$ and $D$) occur at the same value of $t$, and consequently, $A_{t} > B_{t}$ (resp.$C_{t} > D_{t}$) for all values of $t\in \{0, 1, ..., x-1\}$. 

Suppose $$min\{A\} = t'q_{/x} + \alpha y + x,$$ $$min\{B\} = t'q_{/x} + \alpha y - x,$$ $$min\{C\} = t''q_{/x} + \beta y + x,$$ and $$min\{D\} = t''q_{/x} + \beta y - x.$$
This means that for each $t$, $A_{t} - B_{t} = 2x = C_{t} - D_{t}$ because $\frac{\pi (t'q_{/x} + \alpha y + x)}{q}$, $\frac{\pi (t'q_{/x} + \alpha y - x)}{q}$, $\frac{\pi (t''q_{/x} + \beta y + x)}{q}$, and $\frac{\pi (t''q_{/x} + \beta y - x)}{q}$ lie between $0$ and $\frac{\pi}{x}( = \frac{\pi q_{/x}}{q}) $ and there are a total of $x$ such combinations with each $\frac{\pi A_{t}}{q}$ (resp.$\frac{\pi B_{t}}{q}$, $\frac{\pi C_{t}}{q}$, and $\frac{\pi D_{t}}{q}$) lying between $\frac{(t-1)\pi}{x}$ and $\frac{t\pi}{x}$, and is simply a translation of $\frac{\pi A_{t-1}}{q}$ (resp.$\frac{\pi B_{t-1}}{q}$, $\frac{\pi C_{t-1}}{q}$, and $\frac{\pi D_{t-1}}{q}$) by $\frac{\pi}{x}$ to the right.
\smallskip

Using the above remark, we can re-write Equation (\ref{test3}) as
\begin{equation}\label{Cot4}
\sum\limits_{t=0}^{x-1}\Big[\cot\frac{\pi}{q}A_{t'+t} - \cot\frac{\pi}{q}B_{t'+t}\Big] - \Big[\cot\frac{\pi}{q}C_{t''+t} - \cot\frac{\pi}{q}D_{t''+t}\Big] = 0
\end{equation}
\smallskip

Now if $\Big[\cot\frac{\pi}{q}A_{t'} - \cot\frac{\pi}{q}B_{t'}\Big] - \Big[\cot\frac{\pi}{q}C_{t''} - \cot\frac{\pi}{q}D_{t''}\Big] < 0 (\text{ resp.} > 0)$, then $\Big[\cot\frac{\pi}{q}A_{t'+t} - \cot\frac{\pi}{q}B_{t'+t}\Big] - \Big[\cot\frac{\pi}{q}C_{t''+t} - \cot\frac{\pi}{q}D_{t''+t}\Big] < 0 (\text{ resp.} > 0)$ for all values of \textit{t}, which means Equation (\ref{Cot4}) will not be satisfied. So, we conclude that for all values of \textit{t} 
\smallskip
$$\Big[\cot\frac{\pi}{q}A_{t'+t} - \cot\frac{\pi}{q}B_{t'+t}\Big] - \Big[\cot\frac{\pi}{q}C_{t''+t} - \cot\frac{\pi}{q}D_{t''+t}\Big] = 0.$$ 
\smallskip

This means one of the following two conditions must be true:

\begin{itemize}
\item[(I)] $\cot\frac{\pi}{q}A_{t'+t} = \cot\frac{\pi}{q}C_{t''+t}$ and $\cot\frac{\pi}{q}B_{t'+t} = \cot\frac{\pi}{q}D_{t''+t}$, or
\item[(II)] $\cot\frac{\pi}{q}A_{t'+t} = -\cot\frac{\pi}{q}D_{t''+t}$ and $\cot\frac{\pi}{q}B_{t'+t} = -\cot\frac{\pi}{q}C_{t''+t}$
\end{itemize}
\smallskip

Condition (I) implies that $A_{t'+t}\equiv C_{t''+t}(\text{ mod }q)$ and $B_{t'+t}\equiv D_{t''+t}(\text{ mod }q)$, i.e., $\exists t_1, t_2 \in \{0, 1, ..., x-1\}$ with $$pyt_1q_{/x} - py + x \equiv A_{t'}(\text{mod }q),$$ $$pyt_1q_{/x} - py - x \equiv B_{t'}(\text{mod }q),$$ $$syt_2q_{/x} - sy + x \equiv C_{t''}(\text{mod }q),$$ and $$syt_2q_{/x} - sy - x \equiv D_{t''}(\text{mod }q)$$ 
\noindent such that 
$$py(t_1+t)q_{/x} - py + x \equiv sy(t_2+t)q_{/x} - sy + x(\text{ mod }q), \forall t\in\{0, 1,..., x-1\} $$
and
$$py(t_1+t)q_{/x} - py - x \equiv sy(t_2+t)q_{/x} - sy - x(\text{ mod }q), \forall t\in\{0, 1,..., x-1\} $$

These congruences imply 
\begin{equation}\label{eq11}
py[(t_1+t)q_{/x} - 1] \equiv sy[(t_2+t)q_{/x} - 1](\text{ mod }q), \forall t\in\{0, 1,..., x-1\}
\end{equation}
\smallskip

\noindent Now, if $t = x - t_1$, then the above congruence becomes
\begin{equation}\label{eq11}
py(q - 1) \equiv sy(t_3q_{/x} - 1)(\text{ mod }q), \text{ where } t_3 = x - t_1 + t_2.
\end{equation}

We know that $gcd(q-1, q) = 1$. We claim that $gcd(t_3q_{/x} - 1, q) = 1$. To see this, suppose $gcd(t_3q_{/x} - 1, q) = d > 1$. But this means $$py(q - 1) \equiv sy(t_3q_{/x} - 1)(\text{ mod }d)\equiv 0 (\text{ mod }d).$$ 
\noindent Now $d$ does not divide $q_{/x}$ since $d|t_3q_{/x} - 1$, which means $d|x$ since $d|q$. Since $gcd(x, py) = 1$, this would imply that $(q - 1) \equiv 0(\text{ mod }d)$, which is a contradiction. Therefore, $gcd(t_3q_{/x} - 1, q) = 1$. Now we see that the corresponding lens spaces are isometric because 
$$L(q; x, py) \sim L(q; -x, -py) \sim L(q; -x, (t_3q_{/x} - 1)sy) \sim L(q; x, sy).$$
\smallskip

Condition (II) implies that $A_{t'+t}\equiv -D_{t''+t}(\text{ mod }q)$ and $B_{t'+t}\equiv -C_{t''+t}(\text{ mod }q)$, i.e., $\exists t_1, t_2 \in \{0, 1, ..., x-1\}$ with $$pyt_1q_{/x} - py + x \equiv A_{t'}(\text{mod }q),$$ $$pyt_1q_{/x} - py - x \equiv B_{t'}(\text{mod }q),$$ $$syt_2q_{/x} - sy + x \equiv C_{t''}(\text{mod }q),$$ and $$syt_2q_{/x} - sy - x \equiv D_{t''}(\text{mod }q)$$ 
\noindent such that 
$$py(t_1+t)q_{/x} - py + x \equiv -sy(t_2+t)q_{/x} + sy + x(\text{ mod }q), \forall t\in\{0, 1,..., x-1\} $$
and
$$py(t_1+t)q_{/x} - py - x \equiv -sy(t_2+t)q_{/x} + sy - x(\text{ mod }q), \forall t\in\{0, 1,..., x-1\} .$$

These congruences imply 
\begin{equation}\label{eq11}
py[(t_1+t)q_{/x} - 1] \equiv -sy[(t_2+t)q_{/x} + 1](\text{ mod }q), \forall t\in\{0, 1,..., x-1\}
\end{equation}

\noindent As before if $t = x - t_1$, then the above congruence becomes
\begin{equation}\label{eq11}
py(q - 1) \equiv -sy(t_3q_{/x} + 1)(\text{ mod }q), \text{ where } t_3 = x - t_1 + t_2.
\end{equation}

\noindent With a similar argument as in Condition (I), we get that $gcd(t_3q_{/x} + 1, q)= 1$, and, as before, the corresponding lens spaces are isometric because 
$$L(q; x, py) \sim L(q; -x, -py) \sim L(q; -x, -(t_3q_{/x} + 1)sy) \sim L(q; x, sy)$$.

Finally, notice that if $y = q_{/x}$ then $gcd(x, q_{/x}) = 1$ and (\ref{test3}) can be written as 

\begin{align}\label{test4}
&\sum\limits_{t=1}^{x}\Big[\cot\frac{\pi}{q}[tq_{/x} + \alpha q_{/x}  + x] - \cot\frac{\pi}{q}[tq_{/x} + \alpha q_{/x} - x]\Big] \notag \\
= &\sum\limits_{t=1}^{x}\Big[\cot\frac{\pi}{q}[tq_{/x} + \beta q_{/x} + x] - \cot\frac{\pi}{q}[tq_{/x} +\beta q_{/x} - x]\Big],
\end{align}
\noindent which can be re-written as

\begin{align}\label{test5}
&\sum\limits_{t=0}^{x}\Big[\cot\frac{\pi}{q}[tq_{/x} + x] - \cot\frac{\pi}{q}[tq_{/x} - x]\Big] \notag \\
= &\sum\limits_{t=1}^{x}\Big[\cot\frac{\pi}{q}[tq_{/x} + x] - \cot\frac{\pi}{q}[tq_{/x} - x]\Big],
\end{align}

In this case, the minimum positive value for $tq_{/x} + x$ is $x$, which occurs when $t=0$, and the minimum positive value for $tq_{/x} - x$ is $q_{/x} - x$, which occurs when $t = 1$. If $q_{/x} > 2x$ (alt. $q_{/x} < 2x$), then the minimum value of $tq_{/x} - x$ (i.e., $q_{/x} - x$) is greater than (alt. less than) the minimum value of $tq_{/x} + x$ (i.e., $x$). Consequently, $A_t < B_{t+1}$ and $C_t < D_{t+1}$ for all $t\in\{0, 1, ..., x-1\}$ (alt. $A_t > B_{t+1}$ and $C_t > D_{t+1}$ for all $t\in\{0, 1, ..., x-1\}$). This means that for each $t$, $B_{t+1} - A_t = q_{/x} - 2x = D_{t+1} - C_{t}$ (alt. $A_t - B_{t+1} = 2x - q_{/x} = C_{t} - D_{t+1}$).
\smallskip 

We can now re-write equation (\ref{test5}) as
\begin{equation}\label{Cot5}
\sum\limits_{t=0}^{x-1}\Big[\cot\frac{\pi}{q}A_{t} - \cot\frac{\pi}{q}B_{t+1}\Big] - \Big[\cot\frac{\pi}{q}C_{t} - \cot\frac{\pi}{q}D_{t+1}\Big] = 0
\end{equation}

Now if $\Big[\cot\frac{\pi}{q}A_{0} - \cot\frac{\pi}{q}B_{1}\Big] - \Big[\cot\frac{\pi}{q}C_{0} - \cot\frac{\pi}{q}D_{1}\Big] < 0 (\text{ resp.} > 0)$, then $\Big[\cot\frac{\pi}{q}A_{t} - \cot\frac{\pi}{q}B_{t+1}\Big] - \Big[\cot\frac{\pi}{q}C_{t} - \cot\frac{\pi}{q}D_{t+1}\Big] < 0 (\text{ resp.} > 0)$ for all values of \textit{t}, which means equation (\ref{Cot5}) will not be satisfied. So we conclude that for all valuesof \textit{t}, $$\Big[\cot\frac{\pi}{q}A_{t} - \cot\frac{\pi}{q}B_{t+1}\Big] - \Big[\cot\frac{\pi}{q}C_{t} - \cot\frac{\pi}{q}D_{t+1}\Big] = 0.$$ 

\noindent Now the rest of the argument is very similar to the case where $y \neq q_{/x}$.

This completes our proof for Case 5.

\end{proof}

\section{4-Dimensional Orbifold Lens Spaces}\label{HigherLens}

It is known that in the manifold case, even dimensional spherical space forms are only the sphere and the real projective spaces \cite{I2}. It is also known that the sphere $\mathbb{S}^n$ is not isospectral to the real projective space $P^n(\mathbb{R})$ \cite{BGM}. 

In the orbifold case, there are many even dimensional spherical space forms with fixed points. We will focus on the 4-dimensional orbifold lens spaces. In \cite{L}, Lauret has classified cyclic subgroups of $SO(2n + 1)$ up to conjugation. According to this classification, any cyclic subgroup $G$ of $SO(2n + 1)$ is represented by $G = <\gamma>$ where 
$\gamma = diag(R(\frac{2\pi p_1}{q}), ..., R(\frac{2\pi p_n}{q}), 1)$ and $R(\theta ) = \begin{pmatrix} \cos \theta & \sin \theta \\ -
\sin \theta & \cos \theta \end{pmatrix}$.

In order to prove our theorem for $4$-dimensional orbifold lens spaces, we need a couple of results from \cite{Ba}. We define 
\[ \tilde{g}_{W+} = 
\begin{pmatrix}
 R(p_{1} / q) & & & \text{ {\huge 0}} \\
  & \ddots & & \\
  & & R(p_{n} / q)& \\
   \text{ {\huge 0}} & & & I_W
\end{pmatrix}
\]
and
\[ \tilde{g}'_{W+} = 
\begin{pmatrix}
 R(s_{1} / q) & & & \text{ {\huge 0}} \\
  & \ddots & & \\
  & & R(s_{n} / q)& \\
   \text{ {\huge 0}} & & & I_W
\end{pmatrix}
\] \\
where $I_W$ is the $W \times W$ identity matrix for some integer $W$. We can
define $\tilde{G}_{W+}$ $= \langle \tilde{g}_{W+} \rangle$ and 
$\tilde{G}'_{W+} = \langle \tilde{g}'_{W+} \rangle$. Then $\tilde{G}_{W+}$ and
$\tilde{G}'_{W+}$ are cyclic groups of order $q$. We define lens spaces 
$\tilde{L}_{W+} = S^{2n+W-1} / \tilde{G}_{W+}$ and $\tilde{L}'_{W+} = S^{2n+W-1}
/ \tilde{G}'_{W+}$. Further suppose the corresponding $2n-1$-dimensional orbifold 
lens spaces are given by $L = L(q: p_1, p_2, ..., p_n)$ and $L' = L(q: s_1, s_2, ..., s_n)$.
Then by Lemma 3.2.2 in \cite{Ba} we get

\begin{lemma}\label{lemma422}	
	Let $L$, $L'$, $\tilde{L}_{W+}$ and $\tilde{L}'_{W+}$ be as defined above. Then
$L$ is isometric to $L'$ iff $\tilde{L}_{W+}$ is isometric to 
	$\tilde{L}'_{W+}$.
\end{lemma}

\noindent And by Theorem 3.2.3 in \cite{Ba} we get: 

\begin{thm}\label{theorem423}
	Let $F_{q}^{W+}(z: p_1, \ldots, p_n,0)$ be the generating 
	function associated to the spectrum of $\tilde{L}_{W+}$. Then on the domain
$\big \{z \in \mathbf{C} \big \arrowvert \abs{z} < 1 \big \}$, 
	\begin{align*}
		F_{q}^{W+}(z: p_1, \ldots, p_n, 0) = \frac{ (1 + z )}{(1-z)^{W-1}} \cdot
\frac{1}{q} \sum_{l=1}^{q} \frac{1}{\prod_{i = 1}^{n} (z - \gamma^{p_i l})(z -
\gamma^{-p_i l})}
	\end{align*}
\end{thm}

\smallskip

Now suppose $n = 2$. Let 

\[ \tilde{g}_{1} = 
\begin{pmatrix}
R(p_{1} / q) & & & \text{ {\huge 0}} \\
& & R(p_{2} / q)& \\
\text{ {\huge 0}} & & & 1
\end{pmatrix}
\]
\noindent and 
\[ \tilde{g}_{2} =  
\begin{pmatrix}
R(s_1 / q) & & & \text{ {\huge 0}} \\
& & R(s_2/ q)& \\
	\text{ {\huge 0}} & & & 1
	\end{pmatrix}.
	\]
\noindent Suppose there are $4$-dimensional orbifold lens spaces $O_1 = \mathbb{S}^4/\tilde{G}_1$ (denoted by $L(q: p_1, p_2, 0)$) and $O_2 = \mathbb{S}^4/\tilde{G}_2$ (denoted by $L(q: s_1, s_2)$), where $\tilde{G}_1 = <\tilde{g}_1>$ and $\tilde{G}_2 = <\tilde{g}_2>$. Further suppose the corresponding $3$-dimensional orbifold lens spaces are given by $L_1 = L(q: p_1, p_2)$ and $L_2 = L(q: s_1, s_2)$.

We now prove the following theorem for $4$-dimensional orbifold lens spaces:

\begin{thm}\label{thm:7}
Given $O_1$, $O_2$, $\tilde{G}_1$ and $\tilde{G}_2$ as above. If $O_1$ and $O_2$ are isospectral then they are isometric.
\end{thm}

\begin{proof}
From Theorem \ref{theorem423} we know that on the domain $\big \{z \in \mathbf{C} \big \arrowvert \abs{z} < 1 \big \}$, the spectrum generating functions of $O_1$ and $O_2$, respectively, are, 
\begin{align*}
	F_{q}(z: p_1, p_2, 0) = \frac{1}{q} \sum_{l=1}^{q} \frac{(1 + z)}{\prod_{i = 1}^{2} (z - \gamma^{p_i l})(z -
		\gamma^{-p_i l})}
\end{align*}

 and 
 
\begin{align*}
	F_{q}(z: s_1, s_2, 0) = \frac{1}{q} \sum_{l=1}^{q} \frac{(1 + z)}{\prod_{i = 1}^{2} (z - \gamma^{s_i l})(z -
		\gamma^{-s_i l})}
\end{align*}.

Notice that $F_{q}(z: p_1, p_2) = (1 - z)F_{q}(z: p_1, p_2, 0)$ and $F_{q}(z: s_1, s_2) = (1 - z)F_{q}(z: s_1, s_2, 0)$, where $F_{q}(z: p_1, p_2)$ and $F_{q}(z: s_1, s_2)$ are respectively the spectrum generating functions for the 3-dimensional orbifold lens spaces $L_1 = L(q: p_1, p_2)$ and $L_2 = L(q: s_1, s_2)$. This means that if $O_1$ and $O_2$ are isospectral then $L_1$ and $L_2$ are also isospectral. 

Now, from Theorem \ref{thm:6}, we know that $L_1$ and $L_2$ are isometric. By Lemma \ref{lemma422} we know that $L_1$ is isometric to $L_2$ iff $O_1$ is isometric to $O_2$. This proves the theorem.
\end{proof}

\section{Lens Spaces and Other Spherical Space Forms}

One question still remains: Is an orbifold lens space ever isospectral to an orbifold spherical space form which has non-cyclic fundamental group? 

Our next result proves that an orbifold lens space cannot be isospectral to a general spherical space form with non-cyclic fundamental group. We will use some results from \cite{I2} noting that in some cases his assumption that the acting group is fixed-point free is not used in certain proofs, and therefore, the results hold true for orbifolds. 

\begin{definition}
Let $G$ be  finite group, and let $G_k$ be the subset of $G$ consisting of all elements of order $k$ in G. Let $\sigma(G)$ denote the set consisting of orders of elements in $G$. Then we have
\[G = \cup_{k\in\sigma(G)} G_k   \text{         (disjoint union)}\]
\end{definition}

\noindent The following lemma is proved in \cite{I2} for fixed-point free subgroups of $SO(2n)$, but we note that the proof doesn't require this condition and reproduce the proof from \cite{I2}.

\begin{lemma}\label{keylem1}
Let $G$ be a finite subgroup of $SO(2n)\text{ }(n\geq 2)$. Then the subset $G_k$ is divided into the disjoint union of subsets $C^1_k, ..., C^{i_k}_k$ such that each $C^t_k (t = 1, 2, ..., i_k)$ consists of all generic elements of some cyclic subgroup of order $k$ in $G$.
\end{lemma}

\begin{proof}

For any $g\in G_k$, we denote by $A_g$ the cyclic subgroup of G generated by $g$. Now, for $g, g'\in G_k$ the cyclic group $A_g\cap A_{g'}$ is of order $k$ if and only if $A_g = A_{g'}$. Now the lemma follows from this observation immediately. 

\end{proof}

We now state another lemma (see \cite{I2} for proof) that will be used to prove our result.
\begin{lem}\label{keylem2}
Let $g$ be an element in $SO(2n)\text{ }(n\geq 2)$ and of order $q\text{ }(q\geq 3)$. Set $\gamma = e^{2\pi\sqrt{-1}/q}$. Assume $g$ has eigenvalues $\gamma$, $\gamma^{-1}$, $\gamma^{p_1}$, $\gamma^{-p_1}$,..., $\gamma^{p_k}$, $\gamma^{-p_k}$ with multiplicities $l, l, i_1, i_1, ..., i_k, i_k,$ respectively, where $p_1, ..., p_k$ are integers prime to $q$ with 
$p_i\not\equiv \pm p_j (mod q)$ (for $1\leq i<j\leq k$), $p\not\equiv\pm l (mod q)$ (for $i= 1, ... , k$) and $l + i_1 + ...+ i_k = n$. Then the Laurent expansion of the meromorphic function $\frac{1 - z^2}{\text{det }(1_{2n} - gz)}$ at $z = \gamma$ is
\[\frac{1}{(z - \gamma)^l}\frac{({\sqrt{-1}})^{n+l}{\gamma}^l}{2^{n - l}(1 - {\gamma}^2)^{n - 1}}\prod_{j=1}^{k}\{\cot\frac{\pi}{q}(p_j + 1) - \cot\frac{\pi}{q}(p_j - 1)\}^{i_j} + \text{ lower order terms. }\]
\end{lem}

The following proposition is proved by Ikeda for a group $G$ that acts freely. However, we note that the proposition is true even if $G$ does not act freely since the proof does not use the property that $G$ acts freely. 

\begin{prop}\label{keyprop1}
Let $G$ be a finite subgroup of $SO(2n)\text{ } (n\geq 2)$, and let $k \in \sigma(G)$. We define a positive integer $k_0$ by 
\begin{align*}
k_0 &= 2n - 1 \text{     if    } k = 1 \text{ or } 2,\\
    &= max_{g\in G_k} \{\text{max. of multiplicities of eigenvalues of g}\} \text{  if  }k \geq 3.
\end{align*}
Then the generating function $F_G(z)$ has a pole of order $k_0$ at any primitive $k$-th root of 1.
\end{prop}
\begin{proof}
At $z = 1$, we notice that for $g = I_{2n} \in G_1$, we get
\[\lim\limits_{z \rightarrow 1}(1 - z)^{2n - 1}F_G(z) = \frac{2}{|G|},\]
as $g$ has eigenvalue 1 with multiplicity $2n$. So, $F_G(z)$ has a pole of order $2n - 1$ at $z = 1$.

At $z = -1$ we notice that for $g = -I_{2n} \in G_2$, we get 
\[\lim\limits_{z \rightarrow 1}(1 + z)^{2n - 1}F_G(z) = \frac{2}{|G|},\]
as $g$ has eigenvalue -1 with multiplicity $2n$. Also, for any other $g'\in G_2$, the eigenvalue -1 has multiplicity at most $2n$. So $F_G(z)$ has a pole of order $2n - 1$ at $z = -1$ as well.
 
 We now assume $k \geq 3$. Now let $G_k, C^1_k, ..., C^{i_k}_k$ be as in Lemma \ref{keylem1}. Then we have
 \begin{align}\label{equation1}
 \begin{split}
 |G| F_G(z) &= \sum_{g\in G_k}\frac{1 - z^2}{det(I_{2n} - gz)} + \sum_{g\in G - G_k}\frac{1 - z^2}{det(I_{2n} - gz)}\\
 			&= \sum_{j = 1}^{i_k} \sum_{g\in G_k}\frac{1 - z^2}{det(I_{2n} - gz)} + \sum_{g\in G - G_k}\frac{1 - z^2}{det(I_{2n} - gz)}
 \end{split}
 \end{align}
 
 Set $\gamma = e^{2\pi\sqrt{-1}/k}$. For any primitive $k$-th root $\gamma^t$ of 1, where $t$ is an integer prime to $k$, let
 \[\frac{a_{k_0}(t)}{(z - \gamma^t)^{k_0}} + \frac{a_{k_0 - 1}(t)}{(z - \gamma^t)^{k_0 - 1}} + ... + \frac{a_{1}(t)}{(z - \gamma^t)}\]
 be the principal part of the Laurent expansion of $F_G(z)$ at $z = \gamma^t$. Then each coefficient $a_i(t)$ is an element in the $k$-th cyclotomic field $\mathbb{Q}(\gamma)$ over the rational number field $\mathbb{Q}$. The automorphisms $\sigma_t$ of $\mathbb{Q}(\gamma)$ defined by
 \[\gamma\rightarrow\gamma^t\]
 transforms $a_i(1)$ to $a_i(t)$ by Equation (\ref{equation1}). Hence, it is sufficient to show that the generating function $F_G(z)$ has a pole of order $k_0$ at $z - \gamma$, that is, to show that $a_{k_0}(1) \neq 0$.
 
 Note that if $0 < b < a < \pi$, then $\cot a - \cot b < 0$.
 Now the proposition follows immediately from Lemma \ref{keylem2} and Equation (\ref{equation1}).
 \end{proof}
 
 From Proposition \ref{keyprop1}, we get
 
 \begin{cor}\label{cor1}
 Let $\mathbb{S}^{2n-1}/G$ and $\mathbb{S}^{2n-1}/G'$ be two isospectral orbifold spherical space forms. Then $\sigma(G) = \sigma(G')$.
 \end{cor}
 
 We now prove our result
 
 \begin{thm}\label{thm:8}
 Let $\mathbb{S}^{2n-1}/G$ and $\mathbb{S}^{2n-1}/G'$ be two (orbifold) spherical space forms. Suppose $G$ is cyclic and $G'$ is not cyclic. Then $\mathbb{S}^{2n-1}/G$ and $\mathbb{S}^{2n-1}/G'$ cannot be isospectral.
 \end{thm}
 
 \begin{proof}
 By Corollary \ref{corG}, we already know that if $|G| \neq |G'|$ then $\mathbb{S}^{2n-1}/G$ and $\mathbb{S}^{2n-1}/G'$ cannot be isospectral. So let us assume that $|G| = |G'| = q$. 
 
Suppose $\mathbb{S}^{2n-1}/G$ and $\mathbb{S}^{2n-1}/G'$ are isospectral. If $G$ is cyclic then it has an element of order $q$. Now, by Corollary \ref{cor1}, $G'$ must also have an element of order $q$, but since $|G'| = q$, that implies that $G'$ is cyclic, which is not true by assumption, and we arrive at a contradiction. This proves the theorem.
\end{proof}

The above results will complete the classification of the inverse spectral problem on orbifold lens spaces in all dimensions, and also imply that orbifold lens spaces cannot be isospectral to any other spherical space forms.

\section{Heat Kernel For Orbifold Lens Spaces}
In the mathematical study of heat conduction and diffusion, a heat kernel is the fundamental solution to the heat equation on a specified domain with appropriate boundary conditions. It is also one of the main tools in the study of the spectrum of the Laplace operator, and is thus of some auxiliary importance throughout mathematical physics. The heat kernel represents the evolution of temperature in a region whose boundary is held fixed at a particular temperature (typically zero), such that an initial unit of heat energy is placed at a point at time t = 0.

In this section we will show that the coefficients of the asymptotic expansion of the heat trace of the heat kernel are not sufficient to obtain the results in the previous sections. More specifically, if two orbifold lens spaces have the same asymptotic expansion of the heat trace, that does not imply that the two orbifolds are isospectral. 

\begin{definition} Let $M$ be a Riemannian manifold. \textit{A heat kernel}, or alternatively, a fundamental solution to the heat equation, is a function 
	\begin{equation}
	K : (0,\infty)\times M \times M \rightarrow M
	\end{equation}
	
	\noindent that satisfies
	
	\begin{enumerate}
		\item $K(t, x, y)$ is $C^1$ in $t$ and $C^2$ in $x$ and $y$;
		\item $\partial K/\partial t$ + $\Delta_2(K) = 0$, where $\Delta_2$ is the Laplacian with respect to the second variable (i.e., the first space variable);
		\item $	\lim_{t\to 0^+}\int_M K(t, x, y)f(y)dy = f(x)$ for any compactly supported function
		$f$ on $M$.
	\end{enumerate}
\end{definition}
	
	\noindent The heat kernel exists and is unique for compact Riemannian manifolds. Its importance stems from the fact that the solution to the heat equation 
	
	\[
	\frac{\partial u}{\partial t} + \Delta (u) = 0,
	\]
	\[
	u : [0,\infty)\times M \rightarrow \mathbb{R},
	\]
	(where $\Delta$ is the Laplacian with respect to the second variable) with initial
	condition $u(0, x) = f(x)$ is given by
	\begin{equation}\label{heatker}
	u(t, x) = \int_M K(t, x, y)f(y)dy.
	\end{equation}
	
	\noindent If $\left\{\lambda_i\right\}$ is the spectrum of $M$ and $\left\{\zeta_i\right\}$ are the associated eigenfunctions
	(normalized so that they form an orthonormal basis of $L^2(M)$), then we can
	write
	\[K(t, x, y) = \sum_{i}e^{-\lambda_it}\zeta_i(x)\zeta_i(y).
	\]
	
\noindent From this, it is clear that the heat trace,
	\[Z(t) = \sum_{i}e^{-\lambda_it},
	\] is
	a spectral invariant. The heat trace has an asymptotic expansion as $t\rightarrow0+$ :
	\[
	Z(t) = (4\pi t)^{dim(M)/2}\sum_{j=1}^{\infty}a_jt^j ,
	\]
	
\noindent where the $a_j$ are integrals over $M$ of universal homogeneous polynomials in the
	curvature and its covariant derivatives (\cite{MP}, see \cite{Gi2} or \cite{CPR} for details). The first
	few of these are
	
	\[a_0 = vol(M),\] 
	\[a_1 =\frac{1}{6}\int_{M}\tau,\]
	\[a_2 =\frac{1}{360} \int_{M}(5{\tau}^2 - 2|\rho|^2 - 10|R|^2),\]
	
\noindent where $\tau = \sum_{a, b = 1}^{dim(M)}R_{abab}$ is the scalar curvature, $\rho = \sum_{c = 1}^{dim(M)}R_{acbc}$ is the Ricci    tensor, and $R$ is the curvature tensor. The dimension, the volume, and the total scalar curvature are thus completely determined by the spectrum. If $M$ is a surface, then the Gauss-Bonnet Theorem implies that the Euler characteristic of $M$ is also a spectral invariant.

\subsection{Heat Trace Results for Orbifolds}

In the case of a $good$ Riemannian orbifold, Donnelly \cite{D} proved the existence of the heat kernel and also proved the following results:

\begin{thm}\label{thm:0}
Let $f:M\rightarrow M$ be an isometry of a manifold M, with fixed point set $\Omega$.
\begin{itemize}
\item[i.] There is an asymptotic expansion as $t\downarrow 0$

	\[
	\sum_{\lambda} Tr({{f_\lambda}^{\sharp}})e^{t\lambda} \approx \sum_{N\in\Omega} (4\pi t)^{-n/2}\sum_{k=0}^{\infty}{t^k}\int _{N} b_k(f,a)d vol_N(a) ,
	\]
	
\noindent where $N$ is a subset of $\Omega$ (and a submanifold of $M$), $\lambda$ is an eigenvalue of $\Delta$, ${{f_\lambda}^{\sharp}}$ is a linear map from $\lambda$-eigenspace to itself induced by \textit{f}, and the functions $b_k(f,a)$ depend only on the germ of \textit{f} and the Riemannian metric of \textit{M} near the points $a\in N$.

\smallskip

\item[ii.] The coefficients $b_k(f, a)$ are of the form $b_k(f, a) = |det B| b_k^{'}(f, a)$ where $b_k^{'}(f, a)$ is an invariant polynomial in the components of $B = (I - A)^{-1}$ (where $A$ denotes the endomorphism induced by \textit{f} on the fiber of the normal bundle over $a \in N$ ) and the curvature tensor \textit{R} and its covariant derivatives at \textit{a}. 

\noindent In particular,	

\begin{align*}
b_0(f, a) = &|det B|,\\
b_1(f, a) = &|det B|(\frac{\tau}{6} + \frac{1}{6}\rho_kk +\frac{1}{3}R_{iksh}B_{ki}B_{h3} + \frac{1}{3}R_{ikth}B_{kt}B_{hi} -\\ &R_{k\alpha h\alpha}B_{ks}B_{hs}).
\end{align*}

\end{itemize}
\end{thm}
\smallskip

In \cite{DGGW} Donnelly's work is extended to general compact orbifolds, where the heat invariants are expressed in a form that clarifies the asymptotic contributions of each part of the singular set of the orbifold. We will summarise the construction used in \cite{DGGW} in the following remarks before stating their main theorem.

\smallskip

\noindent \textbf{Remarks and Notation:}
\begin{enumerate}
\item An Orbifold \textit{O} was identified with the orbit space $F(O)/O(n)$, where $F(O)$ - a smooth manifold - is the orthonormal frame bundle of \textit{O} and \textit{O(n)} is the orthogonal group, acting smoothly on the right and preserving the fibers. It can be shown that the action of \textit{O(n)} on the frame bundle \textit{F(O)} gives rise to a (Whitney) stratification of \textit{O}. The strata are connected components of the isotropy equivalence classes in \textit{O}. The set of regular points of \textit{O} intersects each connected component $O_0$ of \textit{O} in a single stratum that
constitutes an open dense submanifold of $O_0$. The strata of $O$ are referred as $O$-strata.

\item If $(\tilde{U}, G_U, \pi_U)$ is an orbifold chart on $O$, then it can be shown that the action of $G_U$ on $\tilde{U}$ gives rise to stratifications both of $\tilde{U}$ and of $U$. These are referred to as $\tilde{U}$-strata and $U$-strata, respectively.

\item Let $O$ be a Riemannian orbifold and $(\tilde{U}, G_U, \pi_U)$ an orbifold chart. Let $\tilde{N}$ be a $\tilde{U}$-stratum in $\tilde{U}$. Then it can be shown that all the points in $\tilde{N}$ have the same isotropy group in $G_U$; this group is referred to as the isotropy group of $\tilde{N}$, denoted $Iso(\tilde{N})$.

\item Given a $\tilde{U}$-stratum $\tilde{N}$, denote by $Iso^{max}(\tilde{N})$ the set of all $\gamma\in Iso(\tilde{N})$ such that $\tilde{N}$ is open in the fixed point set $Fix(\gamma)$ of $\gamma$. For $\gamma\in G_U$, it can be shown that each component $W$ of the fixed point set $Fix(\gamma)$ of $\gamma$ (equivalently, the fixed point set of the cyclic group generated by $\gamma$) is a manifold stratified by a collection of $\tilde{U}$-strata, and the strata in $W$ of maximal dimension are open and their union has full measure in $W$. In particular, the union of those $\tilde{U}$-strata $\tilde{N}$ for which $\gamma\in Iso^{max}(\tilde{N})$ has full measure in $Fix(\gamma)$.

\item Let $\gamma$ be an isometry of a Riemannian manifold $M$ and let $\Omega(\gamma)$ denote the set of components of the fixed point set of $\gamma$. Each element of $\Omega(\gamma)$ is a submanifold of $M$. For each non-negative integer $k$, Donnelly \cite{D} defined a real-valued function (cited above), which we temporarily denote $b_k((M, \gamma ), .)$, on the fixed point set of $\gamma$. For each $W\in\Omega(\gamma)$, the restriction of $b_k((M, \gamma ), .)$ to $W$ is smooth. Two key properties of the $b_k$ are:

\begin{itemize}
	\item[(a)] \textit{Locality}. For $a\in W$, $b_k((M, \gamma ), a)$ depends only on the germs at $a$ of the Riemannian metric of $M$ and of the isometry $\gamma$. In particular, if $U$ is a $\gamma$-invariant neighborhood of $a$ in $M$, then $b_k((M, \gamma ), a) = b_k((U, \gamma ), a)$.
	
	\item[(b)] \textit{Universality}. If $M$ and $M'$ are Riemannian manifolds admitting the respective isometries $\gamma$ and $\gamma'$, and if $\sigma : M \rightarrow M'$ is an isometry satisfying $\sigma\circ\gamma = \gamma'\circ \sigma$, then $b_k((M, \gamma ), x) = b_k((M', \gamma' ), \sigma(x))$ for all $x \in Fix(\gamma)$.
\end{itemize}
	
\noindent In view of the locality property, we will usually delete the explicit reference to $M$ and rewrite these functions as $b_k(\gamma , .)$, as they are written in \cite{D}.

\item Let $O$ be an orbifold and let $(\tilde{U}, G_U, \pi_U)$ be an orbifold chart. Let $\tilde{N}$ be a $\tilde{U}$-stratum and let $\gamma\in Iso^{max}(\tilde{N})$. Then $\tilde{N}$ is an open subset of a component of $Fix(\gamma)$ and thus, $b_k(\gamma, .) (= b_k(( \tilde{U}, \gamma ), .)$) is smooth on $\tilde{N}$ for each nonnegative integer $k$. Define a function $b_k(\tilde{N}, .)$ on $\tilde{N}$ by

	\[
	b_k(\tilde{N}, x) = \sum_{\gamma\in Iso^{max}(\tilde{N})}b_k(\gamma, x).
	\]
\end{enumerate}

\smallskip

\begin{definition}
	Let $O$ be a Riemannian orbifold and let $N$ be an $O$-stratum.
	\begin{enumerate}
		\item[(i)] For each nonnegative integer $k$, define a real-valued function $b_k(N, .)$ by
		setting $b_k(N, p) = b_k(\tilde{N}, \tilde{p})$ where $(\tilde{U}, G_U, \pi_U)$ is any orbifold chart about $p$, $\tilde{p}\in {\pi_U}^{-1}(p)$, and $\tilde{N}$ is the $\tilde{U}$-stratum through $\tilde{p}$. 
		\item[(ii)] The Riemannian metric on $O$ induces a Riemannian metric - and thus a volume element - on the manifold $N$. Set
		\[
		I_N: = (4\pi t)^{-dim(N)/2}\sum_{k=0}^{\infty}{t^k}\int _{N} b_k(N,x)d vol_N(x),
		\]
		where $d vol_N$ is the Riemannian volume element.
		\item[(iii)] Set
		\[
		I_0 = (4\pi t)^{-dim(O)/2}\sum_{k=0}^{\infty}a_k (O){t^k},
		\]
		where the $a_k(O)$ (which we will usually write simply as $a_k$) are the familiar heat
		invariants. In particular, $a_0 = vol(O)$, $a_1 = \frac{1}{6}\int_{O}\tau(x) d volO(x)$,
		and so forth. Observe that if $O$ is finitely covered by a Riemannian manifold $M$ (say, $O = G\backslash M$) then $a_k(O) = \frac{1}{|G|}a_k(M)$.
	\end{enumerate}
\end{definition}

We now state the theorem that \cite{DGGW} proved:

\begin{thm}\label{thm:1}
	Let O be a Riemannian orbifold and let $\lambda_1 \leq \lambda_2 \leq ...$ be the
	spectrum of the associated Laplacian acting on smooth functions on O. The heat
	trace $\sum_{j=1}^{\infty}e^{-\lambda_{j}t}$ of O is asymptotic as $t \rightarrow 0^+$ to
	\[
	I_0 + \sum_{N \in S(O)}\frac{I_N}{|Iso(N)|}, 
	\]
	where $S(O)$ is the set of all O-strata, $|Iso(N)|$ is the order of the isotropy at each $p\in N$, and $Iso(p)$ is the conjugacy class of subgroups of $O(n)$. This asymptotic expansion is of the form
	\[
	(4\pi t)^{-dim(O)/2}\sum_{j = 0}^{\infty}c_j t^{j/2} 
	\]
	for some constants $c_j$ .
\end{thm}

\subsection{Heat Kernel For 3-Dimensional Lens Spaces}

We define the normal coordinates for a three-sphere as follows \cite{Iv}: Consider a three-sphere of radius \textit{r},
$$\mathbb{S}^3(r) = \{(v_1, v_2, v_3, v_4) \in \mathbb{R}^4 : (v_1)^2 + (v_2)^2 + (v_3)^2 + (v_4)^2 = r^2\},$$ and let $(R, \psi, \theta, \phi)$ be the spherical coordinates in $\mathbb{R}^4$ where $R \in (0, \infty)$, $\psi \in [0, 2\pi]$, $\theta \in (0, \pi]$ and $\phi \in (0, \pi]$. These coordinates are connected with the standard coordinate system $(u_1, u_2, u_3, u_4)$ in $\mathbb{R}^4$ by the following equations:
\begin{align}\label{eq:1}
\textstyle &u_1 = R\sin\psi\sin\theta\cos\phi, \notag \\
\textstyle &u_2 = R\sin\psi\sin\theta\sin\phi, \notag \\
\textstyle &u_3 = R\sin\psi\cos\theta, \notag \\
\textstyle &u_4 = R\cos\psi.
\end{align}

\noindent The equation of $\mathbb{S}^3(r)$ in these coordinates is $R^2 = r^2$. The functions $x_1 = \psi$, $x_2 = \theta$, and $x_3 = \phi$ provide an internal coordinate system on $\mathbb{S}^3(r)$ (without one point) in which the metric \textit{g} induced on $\mathbb{S}^3(r)$ from $\mathbb{E}^3$ has components $g_{ij}$ such that

\[ (g_{ij}) = 
\begin{pmatrix}
r^2 &&  \text{ {\huge 0}} \\
& r^2\sin^2\psi & \\
\text{ {\huge 0}} && r^2\sin^2\psi\sin^2\theta
\end{pmatrix}.
\]

\noindent \textit{g} induces on $\mathbb{S}^3(r)$ a Riemannian connection $\bigtriangledown$. Using the formula 

\[\Gamma^m_{ij} = \frac{1}{2}g^{ml}[\partial_j g_{il} + \partial_i g_{lj} - \partial_l g_{ji}],\]
\noindent we can calculate the Christoffel symbols, which are as follows:

\noindent $\Gamma^2_{21} = \Gamma^2_{12} = \cot\psi$, $\Gamma^3_{31} = \Gamma^3_{13} = \cot\psi$, $\Gamma^3_{32} = \Gamma^3_{23} = \cot\theta$,
$\Gamma^1_{22} = -\sin\psi\cos\psi$, $\Gamma^1_{33} = -\sin\psi\cos\psi\sin^2\theta$, $\Gamma^2_{33} = -\sin\theta\cos\theta$. All the other symbols are zero.

Now let $\gamma: [0,2\pi] \rightarrow \mathbb{S}^3(r)$ be a path in $\mathbb{S}^3(r)$ such that $x_i\circ\gamma = \pi/2$ for $i = 1, 2$ and $x_3\circ\gamma = id|_{[0,2\pi]}$. Since $\cos\pi/2 = \cot\pi/2 = 0$ and $\sin\pi/2 = 1$ we have $\Gamma^i_{jk}|_{\gamma([0, 2\pi])} = 0$, and consequently, if we take $R = r = 1$, we get $g_{ij} = \delta_i^j$. Therefore, the coordinate system $\{x_1, x_2, x_3\}$ and the frame $\{\partial/\partial x_1, \partial/\partial x_2, \partial/\partial x_3\}$ are normal for $\bigtriangledown$ along the path $\gamma$.

From the Equations (\ref{eq:1}) it is clear that the set $\gamma([0, 2\pi])$ is a circle obtained by intersecting $\mathbb{S}^3(r)$ with the $(v_1, v_2)-$plane $\{v\in\mathbb{R}^4: v_i(p) = 0  \text{ for} i\geq3\}$ in $\mathbb{R}^4$. In fact, we have 
\[\gamma([0,2\pi])  = \{(v_1, v_2, 0, 0)\in \mathbb{R}^4: v_1^2 + v_2^2 = r^2\} =\mathbb{S}^1(r) \times (0, 0).\]

It is clear if \textit{C} is a circle on $\mathbb{S}^3(r)$ obtained by intersecting $\mathbb{S}^3(r)$ by a 2-plane through its origin then there are coordinates on $\mathbb{S}^3(r)$ normal along \textit{C} for the Riemannian connection considered above.

We will assume $r = 1$. Then, using the above normal coordinate system, and the formulas
\begin{align*}
&R^{i}_{jlm} = \partial_l \Gamma^i_{mj} - \partial_m \Gamma^i_{lj} + \Gamma^k_{mj}\Gamma^i_{lk} - \Gamma^k_{lj}\Gamma^i_{km},\\
&R_{abcd} = g_{aj}R^j_{bcd},
\end{align*}
\noindent we calculate the values of the curvature as follows:
\begin{align*}
&R_{1212} = R_{\psi\theta\psi\theta} = \sin^{2}\psi,\\
&R_{1313} = R_{\psi\phi\psi\phi} = \sin^{2}\psi\sin^2\theta,\\
&R_{2323} = R_{\theta\phi\theta\phi} = \sin^{4}\psi\sin^2\theta.
\end{align*}
All other values are zero.  The values of the Ricci tensor, calculated by $\rho_{ab} = R^c_{acb}$, are as follows:
\begin{align*}
\textstyle &\rho_{11} = \rho_{\psi\psi} = 2,\\
\textstyle &\rho_{22} = \rho_{\theta\theta} = 2\sin^2\psi,\\
\textstyle &\rho_{33} = \rho_{\phi\phi} = 2\sin^2\psi\sin^2\theta.
\end{align*}

\noindent All other values are zero. We then calculate the scalar curvature as follows:
\[\tau = g^{\psi\psi}\rho_{\psi\psi} + g^{\theta\theta}\rho_{\theta\theta} + g^{\phi\phi}\rho_{\phi\phi} = 6.\]
\noindent Since $\tau$ is constant all its covariant derivatives, $\tau_{; j}$ are zero. Using $\rho_{ab;m} = \partial_m\rho_{ab} - \rho_{lb}\Gamma^{l}_{ma} - \rho_{al}\Gamma^{l}_{mb}$, we also calculate all the covariant derivatives of the Ricci tensor, which turn out to be zero as well.

Let $e_1 = (1, 0, 0, 0)$, $e_2 = (0, 1, 0, 0)$, $e_3 = (0, 0, 1, 0)$ and $e_4 = (0, 0, 0, 1)$ be the standard basis in $\mathbb{R}^4$. We define the following two subsets:
\[N_a = \Big\{(x, y, 0, 0): x^2 + y^2 = 1\Big\} \subset \mathbb{R}^4 \text{    and   } N_b = \Big\{(0, 0, z, w): z^2 + w^2 = 1\Big\} \subset \mathbb{R}^4.\]

The tangent space $T_{e_1}\mathbb{S}^3$, has basis vectors $\{e_2, e_3, e_4\}$ such that \{$e_2\}$ is a basis for $T_{e_1}N_a$ and $\{e_3, e_4\}$ is a basis for $T_{e_1}N_a^\perp$. Similarly, the tangent space $T_{e_4}\mathbb{S}^3$, has basis vectors $\{e_1, e_2, e_3\}$ such that \{$e_3\}$ is a basis for $T_{e_4}N_b$ and $\{e_1, e_2\}$ is a basis for $T_{e_4}N_b^\perp$. We will now calculate the values for $b_0(f, a)$ and $b_1(f, a)$. Suppose $O = \mathbb{S}^3/G$ is an orbifold lens space where $G = <\gamma>$ and 
\[
\gamma =  \begin{pmatrix} 
R(\frac{\hat{p_1}}{q}) & 0 \\\\                   
0 & R(\frac{\hat{p_2}}{q})
\end{pmatrix}, 
\]

\noindent where $\hat{p_1} \not \equiv \pm \hat{p_2} \, (\text{mod} \, q)$. Suppose  $gcd(\hat{p_1}, q) = q_1$ and $gcd(\hat{p_2}, q) = q_2$, so that $\hat{p_1} = p_1q_1$, $\hat{p_2} = p_2q_2$ and $q = \hat{\alpha} q_1 = \hat{\beta} q_2$. Suppose $gcd(\hat{\alpha}, \hat{\beta}) = g$ so that 
$\hat{\alpha} = \alpha g$,  $\hat{\beta} = \beta g$ and $gcd(\alpha, \beta) = 1$. This means we can write $\gamma$ as 
		
\[
\gamma =  \begin{pmatrix} 
R(\frac{p_1}{\alpha g}) & 0 \\\\                   
0 & R(\frac{p_2}{\beta g})
\end{pmatrix}. 
\]

Now 

\[
\gamma^{\hat\alpha} =  \begin{pmatrix} 
I_2 & 0 \\\\                   
0 & R(\frac{p_2\alpha}{\beta}) 
\end{pmatrix} 
\]

\noindent fixes $N_a$, and 
\[
\gamma^{\hat\beta} =  \begin{pmatrix} 
R(\frac{p_1\beta}{\alpha}) & 0 \\\\                   
0 & I_2
\end{pmatrix} 
\]

\noindent fixes $N_b$, where $I_2$ is the $2\times 2$ identity matrix.

Note that since the group action is transitive and the fixed point sets are $\mathbb{S}^1$, the functions $b_k(.,.)$ are constant along these fixed circles. Therefore, it suffices to consider just a single point in these fixed point sets to calculate the values of the functions. We will choose the points $e_1\in N_a$ and $e_4\in N_b$ to calculate the values of functions. 

We have, in the notation of the Theorem \ref{thm:1}, $\tilde{N_a}\cong\mathbb{S}^1\times\{(0, 0)\}$ 
and $\tilde{N_b}\cong\{(0, 0)\}\times\mathbb{S}^1$.  Also, $Iso_{N_a} = \{1, \gamma^{\hat\alpha}$, $\gamma^{2\hat\alpha}, ... \gamma^{(\beta - 1)\hat\alpha}\}$, 
$|Iso_{N_a}| = \beta$, 
$Iso_{N_b} = \{1, \gamma^{\hat\beta}, \gamma^{2\hat\beta}, ... \gamma^{(\alpha - 1)\hat\beta}\}$ and 
$|Iso_{N_b}| = \alpha.$

We now use Theorem \ref{thm:1} to calculate the heat trace asymptotic for \textit{O} using the formula $I_0 + \frac{I_{N_a}}{\beta} + \frac{I_{N_b}}{\alpha}$ where
\[	I_0 = (4\pi t)^{-dim(O)/2}\sum_{k=0}^{\infty}a_k (O){t^k}
	= (4\pi t)^{-dim(O)/2}\sum_{k=0}^{\infty}\frac{1}{|G|}a_k (\mathbb{S}^3){t^k}\]
\[	= \frac{(4\pi t)^{-3/2}}{q}\sum_{k=0}^{\infty}\frac{\sqrt{\pi}}{4k!}t^k
	= \frac{(4t)^{-3/2}}{4q\pi}\sum_{k=0}^{\infty}\frac{t^k}{k!}
	= \frac{t^{-3/2}}{32q\pi}e^t,\]
and for $i \in {a, b}$, 
\begin{align*}
I_{N_i} &= (4\pi t)^{-dim(N_i)/2}\sum_{k=0}^{\infty}{t^k}\int _{N_i} b_k(N_i,x)d vol_{N_i}(x)\\
      &= \frac{(\pi t)^{-1/2}}{2}\sum_{k=0}^{\infty}{t^k}\int _{\tilde{N_i}} b_k(\tilde{N_i},x)d vol_{\tilde{N_i}}(x) \text{,  since } \tilde{N_i}\rightarrow N_i \text{ is trivial in this case}\\
      &= \frac{(\pi t)^{-1/2}}{2}\sum_{k=0}^{\infty}{t^k}2\pi b_k(\tilde{N_i},x) \text{ (for any choice of \textit{x} by homogeneity)}\\
&= \sqrt{\pi}t^{-1/2}\sum_{k=0}^{\infty}{t^k}b_k(\tilde{N_i},x) \text{  , where  }
b_k(\tilde{N_i},x) = \sum_{\gamma\in Iso^{max}\tilde{N_i}}b_k(\gamma, x).
\end{align*}
Now for $a = e_1$ and $r \in \{1, 2, ...(\beta - 1)\}$,
\begin{align*}
B_{\gamma^{r\hat{\alpha}}}(a) = (I - A_{\gamma^{r\hat{\alpha}}}(a))^{-1} &= \frac{1}{4\sin^2\frac{p_2\pi\alpha r}{\beta}}\begin{pmatrix} 
1 - \cos \frac{2p_2\pi\alpha r}{\beta} & -\sin \frac{2p_2\pi\alpha r}{\beta}\\\\                    
\sin \frac{2p_2\pi\alpha r}{\beta} & 1 - \cos \frac{2p_2\pi\alpha r}{\beta}
\end{pmatrix} \\\\
&= \frac{1}{2}\begin{pmatrix} 
1 & -\cot \frac{p_2\pi\alpha r}{\beta}\\\\                    
\cot \frac{p_2\pi\alpha r}{\beta} & 1
\end{pmatrix}.
\end{align*}

So, $|det B_{\gamma^{r\hat{\alpha}}}(a)| = \frac{1}{4}(1 + \cot^2\frac{p_2\pi\alpha r}{\beta}) = \frac{1}{4\sin^2\frac{p_2\pi\alpha r}{\beta}}$.

Similarly we can show that for $b = e_4$ and $r \in \{1, 2, ...(\alpha - 1)\}$,
\[
B_{\gamma^{r\hat{\beta}}}(b) = \frac{1}{2}\begin{pmatrix} 
1 & -\cot \frac{p_1\pi\beta r}{\alpha}\\\\                    
\cot \frac{p_1\pi\beta r}{\alpha} & 1
\end{pmatrix}, 
\]

and $|det B_{\gamma^{r\hat{\beta}}}(b)| = \frac{1}{4}(1 + \cot^2\frac{p_1\pi\beta r}{\alpha}) = \frac{1}{4\sin^2\frac{p_1\pi\beta r}{\alpha}}$.

We will now calculate $b_i(\tilde{N_j}, .)$ for $i = 0, 1$ and $j = a, b$:

\[
b_0(\gamma^{r\hat{\alpha}}, a) = |det B_{\gamma^{r\hat{\alpha}}}(a)| = \frac{1}{4}(1 + \cot^2\frac{p_2\pi\alpha r}{\beta}) = \frac{1}{4\sin^2\frac{p_2\pi\alpha r}{\beta}}.
\]
So, 
\begin{align*}
 b_0(\tilde{N_a}, a) &= \sum_{f\in Iso^{max}\tilde{N_a}}^{}b_0(f, a)\\
 &= \sum_{r=1}^{\beta - 1}b_0(\gamma^{r\hat{\alpha}}, a) \\
 &= \sum_{r=1}^{\beta - 1}\frac{1}{4}(1 + \cot^2\frac{p_2\pi\alpha r}{\beta}) \\
 &= \sum_{r=1}^{\beta - 1}\frac{1}{4}(1 + \cot^2\frac{\pi r}{\beta}) \text{        , since $gcd(p_2\alpha, \beta) = 1$}\\
 &= \sum_{r=1}^{\beta - 1}\frac{1}{4\sin^2\frac{\pi r}{\beta}}\\
 &= \frac{\beta^2 - 1}{12}  \text{       , by lemma 5.4 in \cite{DGGW}}.
\end{align*}

\noindent We can similarly show that \[b_0(\tilde{N_b}, b) = \sum_{r=1}^{\alpha - 1}\frac{1}{4}(1 + \cot^2\frac{\pi r}{\alpha}) = \frac{\alpha^2 - 1}{12}.\]

We will now calculate $b_1(\tilde{N_a}, a)$ and $b_1(\tilde{N_b}, b)$. Note that for both $B_{\gamma^{r\hat{\alpha}}}(a)$ and $B_{\gamma^{r\hat{\beta}}}(b)$, $B_{13} = B_{23} = B_{31} = B_{32} = B_{33} = 0$. Using the formula in Theorem \ref{thm:0}, we get
\begin{align*}	
b_1(\gamma^{r\hat{\alpha}}, a) = \frac{|det(B_{\gamma^{r\hat{\alpha}}}(a))|}{3}\Big\{&R_{1212}\Big[2 - \frac{1}{4}(\cot\theta_r - \cot\theta_r)^2 - (\frac{1}{2} + \frac{1}{2})^2 -2((\frac{1}{4} + \frac{1}{4})\Big] \\
&+ R_{1313}\Big[2 - (\frac{1}{2} + 0)^2 - 2(\frac{1}{4} +0) - 3(\frac{1}{4}\cot^2\theta_r +0)\Big] \\
&+ R_{2323}\Big[2 - (\frac{1}{2} + 0)^2 - 2(\frac{1}{4} +0) - 3(\frac{1}{4}\cot^2\theta_r +0)\Big] \Big\}, 
\end{align*}
\noindent which gives
\begin{align*}
b_1(\gamma^{r\hat{\alpha}}, a)&= \frac{1}{12}(1 + \cot^2\theta_r)\Big\{R_{1313}\Big(2 - \frac{3}{4} - \frac{3}{4}\cot^2\theta_r)\Big) + R_{2323}\Big(2 - \frac{3}{2} - \frac{3}{4}\cot^2\theta_r)\Big)\Big\}\\
&= \frac{1}{12}(1 + \cot^2\theta_r)(R_{1313} + R_{2323})[2 - \frac{3}{4}(1 + \cot^2\theta_r)]\\
&= (R_{1313} + R_{2323})\Big[\frac{1}{6}(1 + \cot^2\theta_r) - \frac{1}{16}(1 + \cot^2\theta_r)^2\Big]\\
&= (R_{1313} + R_{2323})\Big[\frac{1}{6\sin^2\theta_r} - \frac{1}{16\sin^2\theta_r}\Big],
\end{align*}
\noindent where $\theta_r = \frac{p_2\pi\alpha r}{\beta}$.
 
 So, 
 \begin{align*}
 b_1(\tilde{N_a}, a) &= \sum_{r=1}^{\beta - 1}b_1(\gamma^{r\hat{\alpha}}, a) \\
 &= \sum_{r=1}^{\beta - 1}(R_{1313} + R_{2323})\Big[\frac{1}{6\sin^2\frac{p_2\pi\alpha r}{\beta}} - \frac{1}{16\sin^2\frac{p_2\pi\alpha r}{\beta}}\Big]\\
 &= (R_{1313} + R_{2323})\Big[\frac{1}{6}\sum_{r=1}^{\beta - 1}\frac{1}{\sin^2\frac{\pi r}{\beta}} - \frac{1}{16}\sum_{r=1}^{\beta - 1}\frac{1}{\sin^4\frac{\pi r}{\beta}}\Big],
 \end{align*}
 since $gcd(p_2\alpha, \beta) = 1$. 
 
 \noindent Also, $\sum_{r=1}^{\beta - 1}\frac{1}{\sin^2\frac{\pi r}{\beta}} = \frac{\beta^2 - 1}{3}$ and $\sum_{r=1}^{\beta - 1}\frac{1}{\sin^4\frac{\pi r}{\beta}} = \frac{\beta^4 + 10\beta^2 - 11}{45}$ (see \cite{DGGW}). So we get
 \begin{align*}
 b_1(\tilde{N_a}, a) &= (R_{1313} + R_{2323})\Big(\frac{\beta^2 - 1}{18} - \frac{\beta^4 + 10\beta^2 - 11}{720}\Big)\\
 &= -(R_{1313} + R_{2323})\frac{(\beta^2 - 29)(\beta^2 - 1)}{720}.
\end{align*}

We can similarly show that 
 \begin{align*}
 	b_1(\tilde{N_b}, b) &= (R_{1313} + R_{2323})\Big(\frac{\alpha^2 - 1}{18} - \frac{\alpha^4 + 10\alpha^2 - 11}{720}\Big)\\
 	&= -(R_{1313} + R_{2323})\frac{(\alpha^2 - 29)(\alpha^2 - 1)}{720}.
 \end{align*}
 
 Using Theorem \ref{thm:1} we now calculate the first few coefficients of the asymptotic expansion as follows:
\begin{align*}
&I_0 + \frac{I_{N_a}}{|Iso(N_a)|} + \frac{I_{N_b}}{|Iso(N_b)|}\\
=& \frac{t^{-3/2}}{32q\pi}e^t + \frac{(\pi t)^{-1/2}}{\beta}\Big[t^0\pi b_0(\tilde{N_a}, a) + t^1\pi b_1(\tilde{N_a}, a) + ...\Big]\\
&+ \frac{(\pi t)^{-1/2}}{\alpha}\Big[t^0\pi b_0(\tilde{N_b}, b) + t^1\pi b_1(\tilde{N_b}, b) + ...\Big]\\
=& \frac{t^{-3/2}}{32q\pi}(1 + t + \frac{t^2}{2} + \frac{t^3}{6} + \frac{t^4}{24} + ...) + \Big(\frac {b_0(\tilde{N_a}, a)}{\beta} + \frac {b_0(\tilde{N_b}, b)}{\alpha}\Big)\sqrt{\pi}t^{-1/2}\\
&+ \Big(\frac {b_1(\tilde{N_a}, a)}{\beta} + \frac {b_1(\tilde{N_b}, b)}{\alpha}\Big)\sqrt{\pi}t^{1/2} + ... 
\end{align*}
  
\noindent From this, the coefficient of $t^{-3/2}$ is $\frac{1}{32q\pi}$;\\
\noindent the coefficient of $t^{-1/2}$ is 
 \[
  \frac{1}{32q\pi} + \frac {b_0(\tilde{N_a}, a)}{\beta}\sqrt{\pi} + \frac {b_0(\tilde{N_b}, b)}{\alpha}\sqrt{\pi} = \frac{1}{32q\pi} + \frac{\sqrt{\pi}}{12\beta}(\beta^2 - 1) + \frac{\sqrt{\pi}}{12\alpha}(\alpha^2 - 1);
 \]
\noindent and the coefficient of $t^{1/2}$  is
  \[
  \frac{1}{64q\pi} - \frac{\sqrt{\pi}(R_{1313} + R_{2323})[\alpha(\beta^2 - 29)(\beta^2 - 1) + \beta(\alpha^2 - 29)(\alpha^2 - 1)]}{720\alpha\beta};
  \]

The above results show that the coefficients are dependent on $\alpha$, $\beta$ and the curvature tensor and its covariant derivatives. Since all lens spaces are finitely covered by $\mathbb{S}^3$, the parts of the coefficients that consist of the curvature tensor and its covariant derivatives will be the same for all lens spaces. The only difference will therefore be in the terms containing $\alpha$ and $\beta$. We can rewrite 
\begin{align*}
b_0(\tilde{N_a}, a) &= \sum_{r=1}^{\beta - 1}\frac{1}{4}(1 + \cot^2\frac{p_2\pi\alpha r}{\beta}) = \sum_{r=1}^{\beta - 1}\frac{1}{4} + \sum_{r=1}^{\beta - 1}\frac{1}{4}\cot^2\frac{p_2\pi\alpha r}{\beta}, \\
b_0(\tilde{N_b}, b) &= \sum_{r=1}^{\alpha - 1}\frac{1}{4}(1 + \cot^2\frac{p_1\pi\beta r}{\alpha}) = \sum_{r=1}^{\alpha - 1}\frac{1}{4} + \sum_{r=1}^{\alpha - 1}\frac{1}{4}\cot^2\frac{p_1\pi\beta r}{\alpha},
\end{align*}
\begin{align*}
b_1(\tilde{N_a}, a) =& \sum_{r=1}^{\beta - 1}(R_{1313} + R_{2323})\Big[\frac{1}{6}(1 + \cot^2\frac{p_2\alpha\pi r}{\beta}) - \frac{1}{16}(1 + \cot^2\frac{p_2\alpha\pi r}{\beta})^2\Big]\\
=& \sum_{r=1}^{\beta - 1}\frac{5(R_{1313} + R_{2323})}{48} + \sum_{r=1}^{\beta - 1}\Big(\frac{R_{1313} + R_{2323}}{24}\Big)\cot^2\frac{p_2\alpha\pi r}{\beta}\\
&- \sum_{r=1}^{\beta - 1}\Big(\frac{R_{1313} + R_{2323}}{16}\Big)\cot^4\frac{p_2\alpha\pi r}{\beta},\\
b_1(\tilde{N_b}, b) =& \sum_{r=1}^{\alpha - 1}(R_{1313} + R_{2323})\Big[\frac{1}{6}(1 + \cot^2\frac{p_1\beta\pi r}{\alpha}) - \frac{1}{16}(1 + \cot^2\frac{p_1\beta\pi r}{\alpha})^2\Big]\\
=& \sum_{r=1}^{\alpha - 1}\frac{5(R_{1313} + R_{2323})}{48} + \sum_{r=1}^{\alpha - 1}\Big(\frac{R_{1313} + R_{2323}}{24}\Big)\cot^2\frac{p_1\beta\pi r}{\alpha}\\
&- \sum_{r=1}^{\alpha - 1}\Big(\frac{R_{1313} + R_{2323}}{16}\Big)\cot^4\frac{p_1\beta\pi r}{\alpha},
\end{align*}

\noindent Note that each $b_j(\tilde{N_a}, a)$, $(j = 0, 1)$ is of the form 
\[
b_j(\tilde{N_a}, a) = \sum_{r=1}^{\beta - 1}\sum_{i=1}^{A_j}C^a_{ij}(R)\cot^{\lambda_i}\frac{p_2\alpha\pi r}{\beta},
\]
\noindent where $A_j$ is the finite number of monomials in the powers of $\cot\frac{p_2\alpha\pi r}{\beta}$, and for each $i$, $C^a_{ij}(R)$ are constant functions in terms of the curvature tensor and its covariant derivatives of the covering space, i.e. the sphere. Since $gcd(p_2\alpha, \beta) = 1$, and we are summing over $r$ as it ranges from $1$ to $\beta -1 $, we can write
\[
b_j(\tilde{N_a}, a) = \sum_{r=1}^{\beta - 1}\sum_{i=1}^{A_j}C^a_{ij}(R)\cot^{\lambda_i}\frac{\pi r}{\beta}.
\]
\noindent Similarly, since $gcd(\alpha, p_1\beta) = 1$, we can write
\[
b_j(\tilde{N_b}, b) = \sum_{r=1}^{\alpha - 1}\sum_{i=1}^{A_j}C^b_{ij}(R)\cot^{\lambda_i}\frac{\pi r}{\alpha}.
\]

More generally, for any \textit{k}, the functions $b_k(\gamma^{r\hat{\alpha}}, a)$ and $b_k(\gamma^{r\hat{\beta}}, a)$ are universal polynomials in the components of the curvature tensor, its covariant derivatives and the elements of $B_{\gamma^{r\hat{\alpha}}}(a)$ and $B_{\gamma^{r\hat{\beta}}}(b)$ respectively. Since the elements of $B_{\gamma^{r\hat{\alpha}}}(a)$ are $B_{11} = B_{22} = 1/2$, $B_{12} = -\frac{1}{2}\cot^{\lambda_i}\frac{p_2\alpha\pi r}{\beta}$ and $B_{21} = \frac{1}{2}\cot^{\lambda_i}\frac{p_2\alpha\pi r}{\beta}$, every $b_k(\gamma^{r\hat{\alpha}}, a)$ will be of the form $\sum_{i=1}^{A_j}C^a_{ij}(R)\cot^{\lambda_i}\frac{p_2\alpha\pi r}{\beta}$. This means that for each $k$, we will have, 
\[
b_k(\tilde{N_a}, a) = \sum_{r=1}^{\beta - 1}\sum_{i=1}^{A_k}C^a_{ik}(R)\cot^{\lambda_i}\frac{\pi r}{\beta},
\]
and similarly, 
\[
b_k(\tilde{N_b}, b) = \sum_{r=1}^{\alpha - 1}\sum_{i=1}^{A_k}C^b_{ik}(R)\cot^{\lambda_i}\frac{\pi r}{\alpha}.
\]
This observation gives us the following lemma for three-dimensional orbifold lens spaces:
\begin{lemma}
Given two orbifold lens spaces $O_1 = \mathbb{S}^3/G_1$ and $O_2 = \mathbb{S}^3/G_2$, such that $G_1 = <\gamma_1>$ and $G_2 = <\gamma_2>$ where 
\[
\gamma_1 =  \begin{pmatrix} 
R({\frac{\hat{p_1}}{q}}) & 0 \\\\                    
0 & R({\frac{\hat{p_2}}{q}})
\end{pmatrix}
\]

\noindent with $\hat{p_1} \not \equiv \pm \hat{p_2} \, (\text{mod} \, q)$, $gcd(\hat{p_1}, q) = q_{11}$, $gcd(\hat{p_2}, q) = q_{21}$, $\hat{p_1} = p_1 q_{11}$, $\hat{p_2} = p_2 q_{21}$, $q = \hat{\alpha_1} q_{11} = \hat{\beta_1} q_{21}$, $gcd(\hat{\alpha_1}, \hat{\beta_1}) = g_1$, $\hat{\alpha_1} = \alpha_1 g_1$,  $\hat{\beta_1} = \beta_1 g_1$, and 
\[
\gamma_2 =  \begin{pmatrix} 
R({\frac{\hat{s_1}}{q}}) & 0 \\\\                    
0 & R({\frac{\hat{s_2}}{q}})
\end{pmatrix}, 
\]
\noindent with $\hat{s_1} \not \equiv \pm \hat{s_2} \, (\text{mod} \, q)$, $gcd(\hat{s_1}, q) = q_{12}$, $gcd(\hat{s_2}, q) = q_{22}$, $\hat{s_1} = s_1q_{12}$, $\hat{s_2} = s_2q_{22}$, $q = \hat{\alpha_2} q_{12} = \hat{\beta_2} q_{22}$, $gcd(\hat{\alpha_2}, \hat{\beta_2}) = g_2$, $\hat{\alpha_2} = \alpha_2 g_2$,  $\hat{\beta_2} = \beta_2 g_2$.\\
Then $O_1 = \mathbb{S}^3/G_1$ and $O_2 = \mathbb{S}^3/G_2$ will have the exact same asymptotic expansion of the heat kernel if $\alpha_1 = \alpha_2$ and $\beta_1 = \beta_2$.
\end{lemma}
This lemma gives us a tool to find examples of 3-dimensional orbifold 
lens spaces that are non-isometric (hence non-isospectral) but have 
the exact same asymptotic expansion of the heat kernel. 
\begin{ex}
Suppose $q = 195$, and consider the two lens spaces $O_1 = L(195: 3, 5)$ and $O_2 = L(195: 6, 35)$. Since there is no integer $l$ coprime to $195$ and no $e_i \in \{1, -1 \}$ such that $\{e_1 l 3, e_2 l 5\}$ is a permutation of $\{6, 35\}(\text{mod } q)$, $O_1$ and $O_2$ are not isometric (and hence non-isospectral). However, in the notation of the lemma above, $\hat{p_1} = 3$, $\hat{p_2} = 5$, $\hat{s_1} = 6$, $\hat{s_2} = 35$, $gcd(\hat{p_1}, q) = 3 = gcd(\hat{s_1}, q)$, $gcd(\hat{p_2}, q) = 5 = gcd(\hat{s_2}, q)$ and $q = 195 = 3\times 65 = 5\times 39$. So, $\hat{\alpha_1} = \hat{\alpha_2} = 65$ and $\hat{\beta_1} = \hat{\beta_2} = 39$, with $gcd(\hat{\alpha_i}, \hat{\beta_i}) = 13$ (for $i = 1, 2$) giving  $\alpha_1 = \alpha_2 = 5$ and $\beta_1 = \beta_2 = 3$. Therefore, $O_1 = L(195: 3, 5)$ and $O_2 = L(195: 6, 35)$ have the exact same asymptotic expansion.
\end{ex}

\subsection{Heat Kernel For 4-Dimensional Lens Spaces}

Similar to the three-dimensional case we can show the construction of examples in four-dimensional lens spaces where the lens spaces will not be isospectral but will have the exact same asymptotic expansion of the trace of the heat kernel. We define the normal coordinates for a four-sphere as follows \cite{Iv}: Consider a four-sphere of radius \textit{r},
$$\mathbb{S}^4(r) = \{(v_1, v_2, v_3, v_4, v_5) \in \mathbb{R}^5 : (v_1)^2 + (v_2)^2 + (v_3)^2 + (v_4)^2 + (v_5)^2 = r^2\},$$ 
and let $(R, \psi, \theta, \phi, t)$ be the spherical coordinates in $\mathbb{R}^5$ where $R \in (0, \infty)$, 
$\psi \in (0, \pi]$, $\theta \in (0, \pi]$, $\phi \in (0, \pi]$ and $t \in [0, 2\pi]$. These coordinates are connected with the standard coordinate system $(u_1, u_2, u_3, u_4, u_5)$ in $\mathbb{R}^5$ by the following equations:
\begin{align}\label{eq:1a}
\textstyle &u_1 = R\sin\psi\sin\theta\sin\phi\sin t, \notag \\
\textstyle &u_2 = R\sin\psi\sin\theta\sin\phi\cos t, \notag \\
\textstyle &u_3 = R\sin\psi\sin\theta\cos\phi, \notag \\
\textstyle &u_4 = R\sin\psi\cos\theta, \notag \\
\textstyle &u_5 = R\cos\psi.
\end{align}
\noindent The equation of $\mathbb{S}^4(r)$ in these coordinates is $R^2 = r^2$. The functions $x_1 = \psi$, $x_2 = \theta$, $x_3 = \phi$ and $x_4 = t$ provide an internal coordinate system on $\mathbb{S}^4(r)$ (without one point) in which the metric \textit{g} induced on $\mathbb{S}^4(r)$ from $\mathbb{E}^3$ has components $g_{ij}$ such that

\[ (g_{ij}) = 
\begin{pmatrix}
r^2 &&&  \text{ {\huge 0}} \\
& r^2\sin^2\psi && \\
&& r^2\sin^2\psi\sin^2\theta &\\
\text{ {\huge 0}} &&& r^2\sin^2\psi\sin^2\theta\sin^2\phi
\end{pmatrix}.
\]

\noindent As before, we calculate the values of the curvature tensor as follows:
\begin{align*}
&R_{1212} = R_{\psi\theta\psi\theta} = \sin^{2}\psi,\\
&R_{1313} = R_{\psi\phi\psi\phi} = \sin^{2}\psi\sin^2\theta,\\
&R_{1414} = R_{\psi t \psi t} = \sin^{2}\psi\sin^2\theta\sin^2\phi,\\
&R_{2323} = R_{\theta\phi\theta\phi} = \sin^{4}\psi\sin^2\theta,\\
&R_{2424} = R_{\theta t \theta t} = \sin^{4}\psi\sin^2\theta\sin^2\phi,\\
&R_{3434} = R_{\phi t \phi t} = \sin^{4}\psi\sin^4\theta\sin^2\phi.
\end{align*}
All other values are zero.  The values of the Ricci tensor, calculated by $\rho_{ab} = R^c_{acb}$, are as follows:
\begin{align*}
\textstyle &\rho_{11} = \rho_{\psi\psi} = 3,\\
\textstyle &\rho_{22} = \rho_{\theta\theta} = 3\sin^2\psi,\\
\textstyle &\rho_{33} = \rho_{\phi\phi} = 3\sin^2\psi\sin^2\theta,\\
\textstyle &\rho_{44} = \rho_{tt} = 3\sin^2\psi\sin^2\theta\sin^2\phi.
\end{align*}

\noindent All other values are zero. We then calculate the scalar curvature as follows:
\[\tau = g^{\psi\psi}\rho_{\psi\psi} + g^{\theta\theta}\rho_{\theta\theta} + g^{\phi\phi}\rho_{\phi\phi} + g^{tt}\rho_{tt}= 12.\]

Now, let $e_1 = (1, 0, 0, 0, 0)$, $e_2 = (0, 1, 0, 0, 0)$, $e_3 = (0, 0, 1, 0, 0)$, $e_4 = (0, 0, 0, 1, 0)$ and $e_5 = (0, 0, 0, 0, 1)$ be the standard basis in $\mathbb{R}^5$. We can then define the following two subsets:
\[N_a = \Big\{(x, y, 0, 0, v): x^2 + y^2 + v^2 = 1\Big\} \subset \mathbb{R}^5 \]
\text{    and   } 
\[N_b = \Big\{(0, 0, z, w, v): z^2 + w^2 + v^2 = 1\Big\} \subset \mathbb{R}^5.\]

The tangent space $T_{e_1}\mathbb{S}^4$, has basis vectors $\{e_2, e_3, e_4, e_5\}$ such that \{$e_2, e_5\}$ is a basis for $T_{e_1}N_a$ and $\{e_3, e_4\}$ is a basis for $T_{e_1}N_a^\perp$. Similarly, the tangent space $T_{e_4}\mathbb{S}^4$, has basis vectors $\{e_1, e_2, e_3, e_5\}$ such that \{$e_3, e_5\}$ is a basis for $T_{e_4}N_b$ and $\{e_1, e_2\}$ is a basis for $T_{e_4}N_b^\perp$.

Suppose $O = \mathbb{S}^4/G$ is an orbifold lens space where $G = <\gamma>$ and 
\[
\gamma =  \begin{pmatrix} 
R(\frac{\hat{p_1}}{q}) &&& \text{ {\huge 0}} \\\\                    
& R(\frac{\hat{p_2}}{q}) \\\\                    
\text{ {\huge 0}} &&& 1 \end{pmatrix}, 
\]

\noindent where $\hat{p_1} \not \equiv \pm \hat{p_2} \, (\text{mod} \, q)$. Suppose  $gcd(\hat{p_1}, q) = q_1$ and $gcd(\hat{p_2}, q) = q_2$, so that $\hat{p_1} = p_1q_1$, $\hat{p_2} = p_2q_2$ and $q = \hat{\alpha} q_1 = \hat{\beta} q_2$. Suppose $gcd(\hat{\alpha}, \hat{\beta}) = g$ so that 
$\hat{\alpha} = \alpha g$,  $\hat{\beta} = \beta g$ and $gcd(\alpha, \beta) = 1$. This means we can write $\gamma$ as 
		
\[
\gamma =  \begin{pmatrix} 
R(\frac{p_1}{\alpha g}) &&& \text{ {\huge 0}} \\\\                    
& R(\frac{p_2}{\beta g}) \\\\   
\text{ {\huge 0}} &&& 1 \end{pmatrix}. 
\]

Now 

\[
\gamma^{\hat\alpha} =  \begin{pmatrix} 
I_2 &&& \text{ {\huge 0}} \\\\                    
& R(\frac{p_2 \alpha}{\beta}) \\\\   
\text{ {\huge 0}} &&& 1 \end{pmatrix}
\]

\noindent fixes $N_a$, and 
\[
\gamma^{\hat\beta} =  \begin{pmatrix} 
R(\frac{p_1 \beta}{\alpha}) && \text{ {\huge 0}} \\\\ 
\text{ {\huge 0}} && I_3 \end{pmatrix}
\]

\noindent fixes $N_b$. Here $I_2$ and $I_3$ are the $2\times 2$ and $3\times 3$ identity matrices respectively.

As before, it suffices to consider just a single point in these fixed point sets to calculate the values of the functions. We will choose the points $e_1\in N_a$ and $e_4\in N_b$ to calculate the values of functions. 

We have, in the notation of the Theorem \ref{thm:1}, $\tilde{N_a}\cong\mathbb{S}^2\times\{(0, 0)\}$ 
and $\tilde{N_b}\cong\{(0, 0)\}\times\mathbb{S}^2$.  Also, $Iso_{N_a} = \{1, \gamma^{\hat\alpha}, \gamma^{2\hat\alpha}, ... \gamma^{(\beta - 1)\hat\alpha}\}$, 
$|Iso_{N_a}| = \beta$, 
$Iso_{N_b} = \{1, \gamma^{\hat\beta}, \gamma^{2\hat\beta}, ... \gamma^{(\alpha - 1)\hat\beta}\}$ and 
$|Iso_{N_b}| = \alpha.$

Now, as in the case of three-dimensional lens spaces, we have for $a = e_1$ and $r \in \{1, 2, ...(\beta - 1)\}$,
\begin{align*}
B_{\gamma^{r\hat{\alpha}}}(a) &= \frac{1}{2}\begin{pmatrix} 
1 & -\cot \frac{p_2\pi\alpha r}{\beta}\\\\                    
\cot \frac{p_2\pi\alpha r}{\beta} & 1
\end{pmatrix}.
\end{align*}

\noindent So, $|det B_{\gamma^{r\hat{\alpha}}}(a)| = \frac{1}{4}(1 + \cot^2\frac{p_2\pi\alpha r}{\beta}) = \frac{1}{4\sin^2\frac{p_2\pi\alpha r}{\beta}}$.

Similarly we can show that for $b = e_4$ and $r \in \{1, 2, ...(\alpha - 1)\}$,
\[
B_{\gamma^{r\hat{\beta}}}(b) = \frac{1}{2}\begin{pmatrix} 
1 & -\cot \frac{p_1\pi\beta r}{\alpha}\\\\                    
\cot \frac{p_1\pi\beta r}{\alpha} & 1
\end{pmatrix}, 
\]

\noindent and $|det B_{\gamma^{r\hat{\beta}}}(b)| = \frac{1}{4}(1 + \cot^2\frac{p_1\pi\beta r}{\alpha}) = \frac{1}{4\sin^2\frac{p_1\pi\beta r}{\alpha}}$. Note again that for both $B_{\gamma^{r\hat{\alpha}}}(a)$ and $B_{\gamma^{r\hat{\beta}}}(b)$, $B_{13} = B_{23} = B_{31} = B_{32} = B_{33} = B_{41} = B_{14} = B_{42} = B_{24} = B_{43} = B_{34} = B_{44} = 0$. This means that, just as in the case of three-dimensional lens spaces, for each $k$, we will have, 
\[
b_k(\tilde{N_a}, a) = \sum_{r=1}^{\beta - 1}\sum_{i=1}^{A_k}C^a_{ik}(R)\cot^{\lambda_i}\frac{\pi r}{\beta},
\]
and 
\[
b_k(\tilde{N_b}, b) = \sum_{r=1}^{\alpha - 1}\sum_{i=1}^{A_k}C^b_{ik}(R)\cot^{\lambda_i}\frac{\pi r}{\alpha}.
\]
Similar to the three-dimensional case, this observation gives us the following lemma:
\smallskip
\begin{lemma}
Given two orbifold lens spaces $O_1 = \mathbb{S}^4/G_1$ and $O_2 = \mathbb{S}^4/G_2$, such that $G_1 = <\gamma_1>$ and $G_2 = <\gamma_2>$ where 
\[
\gamma_1 =  \begin{pmatrix} 
R(\frac{\hat{p_1}}{q}) && \text{ {\huge 0}} \\\\                    
& R(\frac{\hat{p_2}}{q}) \\\\
\text{ {\huge 0}} && 1 \\\\
\end{pmatrix}
\]
with $\hat{p_1} \not \equiv \pm \hat{p_2} \, (\text{mod} \, q)$, $gcd(\hat{p_1}, q) = q_{11}$, $gcd(\hat{p_2}, q) = q_{21}$, $\hat{p_1} = p_1 q_{11}$, $\hat{p_2} = p_2 q_{21}$, $q = \hat{\alpha_1} q_{11} = \hat{\beta_1} q_{21}$, $gcd(\hat{\alpha_1}, \hat{\beta_1}) = g_1$, $\hat{\alpha_1} = \alpha_1 g_1$,  $\hat{\beta_1} = \beta_1 g_1$, and 
\[
\gamma_2 =  \begin{pmatrix} 
R(\frac{\hat{s_1}}{q}) && \text{ {\huge 0}} \\\\                    
& R(\frac{\hat{s_2}}{q}) \\\\
\text{ {\huge 0}} && 1 \\\\
\end{pmatrix}, 
\]
\noindent with $\hat{s_1} \not \equiv \pm \hat{s_2} \, (\text{mod} \, q)$, $gcd(\hat{s_1}, q) = q_{12}$, $gcd(\hat{s_2}, q) = q_{22}$, $\hat{s_1} = s_1q_{12}$, $\hat{s_2} = s_2q_{22}$, $q = \hat{\alpha_2} q_{12} = \hat{\beta_2} q_{22}$, $gcd(\hat{\alpha_2}, \hat{\beta_2}) = g_2$, $\hat{\alpha_2} = \alpha_2 g_2$,  $\hat{\beta_2} = \beta_2 g_2$.\\
Then $O_1 = \mathbb{S}^4/G_1$ and $O_2 = \mathbb{S}^4/G_2$ will have the exact same asymptotic expansion of the heat kernel 
if $\alpha_1 = \alpha_2$ and $\beta_1 = \beta_2$.
\end{lemma}

This lemma gives us a tool to find examples of 4-dimensional orbifold 
lens spaces that are non-isometric (hence non-isospectral) but have 
the exact same asymptotic expansion of the heat kernel.

\begin{ex}
Suppose $q = 195$, and consider the two lens spaces $O_1 = \tilde{L}_{1+} = L(195: 3, 5, 0)$ and $O_2 = \tilde{L}'_{1+} = L(195: 6, 35, 0)$ (using the notation from Lemma \ref{lemma422}). Since there is no integer $l$ coprime to $195$ and no $e_i \in \{1, -1 \}$ such that $\{e_1 l 3, e_2 l 5\}$ is a permutation of $\{6, 35\}(\text{mod } q)$, $O_1$ and $O_2$ are not isometric (and hence non-isospectral). However, in the notation of the lemma above, $\hat{p_1} = 3$, $\hat{p_2} = 5$, $\hat{s_1} = 6$, $\hat{s_2} = 35$, $gcd(\hat{p_1}, q) = 3 = gcd(\hat{s_1}, q)$, $gcd(\hat{p_2}, q) = 5 = gcd(\hat{s_2}, q)$ and $q = 195 = 3\times 65 = 5\times 39$. So, $\hat{\alpha_1} = \hat{\alpha_2} = 65$ and $\hat{\beta_1} = \hat{\beta_2} = 39$, with $gcd(\hat{\alpha_i}, \hat{\beta_i}) = 13$ (for $i = 1, 2$) giving  $\alpha_1 = \alpha_2 = 5$ and $\beta_1 = \beta_2 = 3$. Therefore, $O_1$ and $O_2$ have the exact same asymptotic expansion.
\end{ex}











\addtocontents{toc}{\protect\vspace{2ex}}


\begin{thebibliography}{xx}

\bibitem[ALR]{ALR} A. Adem, J. Leida, Y. Ruan, {\it Orbifolds and String Topology}, Cambridge Tracts in Mathematics 171 Cambridge University Press, 2007.

\bibitem[Ba]{Ba} N.Bari {\it Orbifold lens spaces that are isospectral but not isometric}, 
Osaka J.Math, \textbf{48:1} (2011), 1-40.

\bibitem[BCDS]{BCDS} P. Buser, J. Conway, P. Doyle and K. Semmler, {\it Some planar isospectral domains},
Internat. Math. Res. Notices. {\bf 9} (1994), 391ff., approx. 9 pp. (electronic).

\bibitem[BGM]{BGM} M. Berger, P. Gaudachon and E. Mazet, {\it Le spectre d'une
vari\'et\'e riemannienne}, Lecture notes in Mathematics 194, 
Springer-Verlag, Berlin-Heidelberg-New York, 1971. 

\bibitem[BW]{BW} P. B\'erard and D. Webb, {\it On ne peut pas entendre l\'orientabilit\'e d\'une surface}, 
C. R. Acad. Sci. Paris S´er. I Math. {\bf 320} (1995), no. 5, 533-536.

\bibitem[CPR]{CPR} M. Craioveanu, M. Puta, T. Rassias, {\it Old and new aspects in
spectral geometry. Mathematics and applications.} Kluwer Academic,
Dordrecht; London, 2001.

\bibitem[Chi]{Chi} Chiang, Yuan-Jen, {\it Spectral Geometry of V-Manifolds and
its Application to Harmonic Maps}, Proc. Symp. Pure Math. {\bf 54} part 1
(1993), 93--99. 

\bibitem[D]{D} H. Donnelly, {\it Spectrum and the fixed point sets of isometries I}, Math. Ann. 224 (1976), 161-170.

\bibitem[DGGW]{DGGW} E. Dryden, C. Gordon, S. Greenwald and D. Webb, 
{\it Asymptotic expansion of the heat kernel for orbifolds}, Michigan Math J. 56 (2008), 205--238.

\bibitem[DHVW]{DHVW} L. Dixon, J.A. Harvey, C. Vafa, E. Witten, {\it Strings On Orbifolds}, Nuclear Physics B261 (1985) 678 - 686.

\bibitem[DR]{DR} P. Doyle and J. Rossetti, {\it Isospectral hyperbolic surfaces having matching geodesics},
preprint, ArXiv math.DG/0605765.

\bibitem[DV]{DV} P.Du Val, {\it Homographies, Quaternions and Rotations}, Oxford Math.Monographs, Oxford University Press, 1964.

\bibitem[GP]{GP} O. Grosek and S. Porubsky, {\it Coprime solutions to $ax \equiv b (\text{mod }n)$}, J.Math. Cryptol. \textbf{7}(2013), 217--224.

\bibitem[GR]{GR} C. S. Gordon and J. Rossetti, {\it Boundary volume and length spectra of Riemannian
manifolds: what the middle degree Hodge spectrum doesn't reveal}, Ann. Inst. Fourier, {\bf 53}
(2003), no. 7, 2297--2314.

\bibitem[Gi]{Gi} P. B. Gilkey, {\it On spherical space forms with meta-cyclic fundamental group which 
are isospectral but not equivariant cobordant}. Compositio Mathematica, {\bf 56} no. 2 (1985), p. 171-200 

\bibitem[Gi2]{Gi2} P. B. Gilkey, {\it Invariance theory, the heat equation, and the Atiyah-Singer
index theorem.} Publish or Perish, Boston, 1984.

\bibitem[GoM]{GoM} R. Gornet and J. McGowan, {\it Lens spaces, isospectral on forms but not on functions}, 
London Math. Soc. J. of Computation {\bf 9} (2006)  270--286.
    
\bibitem[I1]{I1} A. Ikeda, {\it On lens spaces which are isospectral but not
isometric}, Ann. scient. \'Ec. Norm. Sup. $4^{e}$ s\'eries, t. 13, 303--315.

\bibitem[I2]{I2} A. Ikeda, {\it On the spectrum of a riemannian manifold of
positive constant curvature}, Osaka J. Math., {\it 17} (1980), 75--93.

\bibitem[IY]{IY} A. Ikeda and Y. Yamamoto, {\it On the spectra of a
3-dimensional lens space}, Osaka J. Math., {\it 16} (1979), 447--469.

\bibitem[Iv]{Iv} B. Z. Iliev {\it Handbook of Normal Frames and Coordinates}, 51-55.

\bibitem[K]{K} M. Kac, {\it Can one hear the shape of a drum?}, Amer. Math.
Monthly 73 (1966) no. 4m Part II, 1--23.

\bibitem[L]{L} E.A.Lauret, {\it Spectra of orbifolds with cyclic fundamental groups},
Annals of Global Analysis and Geometry, October, 2015

\bibitem[MP]{MP} S. Minakshisundaram and A. Pleijel, {\it Some properties of the eigen-
functions of the Laplace-operator on Riemannian manifolds.} Canadian
Journal of Mathematics, 1:242-256, 1949.

\bibitem[M]{M} J. Milnor, {\it Eigenvalues of the Laplace operator on certain manifolds}, 
Proc. Nat. Acad. Sci. USA {\bf 51} (1964), 542.

\bibitem[PS]{PS} E. Proctor and E. Stanhope, {\it An Isospectral Deformation on an
Orbifold Quotient of a Nilmanifold}, Preprint, ArXiv math. 0811.0794

\bibitem[RSW]{RSW} J. Rossetti, D. Schueth and M. Weilandt, {\it Isospectral orbifolds 
with different maximal isotropy orders}, Ann. Glob. Anal. Geom. {\bf 34} (2008), 351 - 366

\bibitem[SSW]{SSW} N. Shams, E. Stanhope, and D. Webb, {\it One Cannot Hear
Orbifold Isotropy Type}, Archiv der Math (Basel) 87 (2006), no.4, 375-384.

\bibitem[S1]{S1} E. Stanhope, {\it Hearing Orbifold Topology}, Ph.D. Thesis,
Dartmouth College, 2002. 

\bibitem[S2]{S2} E. Stanhope, {\it Spectral bounds on orbifold isotropy}, Annals
of Global Analysis and Geometry {\bf 27} (2005), no. 4, 355--375.

\bibitem[V]{V} M. F. Vign\'eras, {\it Vari\'et\'es Riemanniennes isospectrales et non isom´etriques}, 
Ann. of Math. {\bf 112} (1980), 21–32.

\bibitem[Y]{Y} Y. Yamamoto, {\it On The Number Of Lattice Points In The Square $|x|+|y|\leq u$ With A Certain
Congruence Condition}, Osaka J. Math., {\it 17} (1980), 9--21.
\end{thebibliography}
\end{document}